\documentclass[a4paper,11pt]{amsart}
\usepackage{amssymb,graphicx}
\usepackage{amsmath,amscd}

\newtheorem{lemma}{Lemma}[section]
\newtheorem{notation}{Notation}
\newtheorem{prop}{Proposition}[section]
\newtheorem{thm}{Theorem}
\newtheorem{cor}{Corollary}
\newtheorem{conj}{Conjecture}
\theoremstyle{definition}
\newtheorem{rem}{Remark}[section]
\newtheorem{defn}{Definition}[section]

\newcounter{numl}
\newtheorem{assumption}{Assumption}

\newcommand{\labelnuml}{\textup{(\roman{numl})}}
\newenvironment{numlist}{\begin{list}{\labelnuml}%
{\usecounter{numl}\setlength{\leftmargin}{0pt}%
\setlength{\itemindent}{2\parindent}%
\setlength{\itemsep}{\smallskipamount}\def
\makelabel ##1{\hss \llap {\upshape ##1}}}}{\end{list}}

\newcommand{\R}{{\mathbb R}}
\newcommand{\PP}{{\mathbb P}}
\newcommand{\C}{{\mathbb C}}
\newcommand{\Z}{{\mathbb Z}}
\newcommand{\N}{{\mathbb N}}
\newcommand{\Q}{{\mathbb Q}}
\newcommand{\T}{{\mathbb T}}
\newcommand{\cP}{{\mathcal P}}
\newcommand{\cE}{{\mathcal E}}
\newcommand{\cL}{{\mathcal L}}
\newcommand{\cM}{{\mathcal M}}
\newcommand{\cO}{{\mathcal O}}

\newcommand{\mult}{^{\scriptscriptstyle\times}}

\newcommand{\rk}{\mathop{\mathrm{rk}}\nolimits}

\newcommand{\restr}[1]{|_{#1}^{\vphantom x}}

\newcommand{\tr}{\mathrm{tr}}

\usepackage{todonotes} 
\usepackage{marginnote}  
\usepackage{enumitem} 
\usepackage[colorlinks=true]{hyperref}  

\allowdisplaybreaks 
\makeatletter
\newcommand{\pushright}[1]{\ifmeasuring@#1\else\omit\hfill$\displaystyle#1$\fi\ignorespaces}  
\newcommand{\pushleft}[1]{\ifmeasuring@#1\else\omit$\displaystyle#1$\hfill\fi\ignorespaces}
\makeatother

\begin{document}
\title[Relative K-polystability  of projective bundles]{Relative K-polystability of projective bundles over a curve}
\date{}
\author[V. Apostolov]{Vestislav Apostolov} \address{Vestislav Apostolov \\
D{\'e}partement de Math{\'e}matiques\\ UQAM\\ C.P. 8888 \\ Succ. Centre-ville
\\ Montr{\'e}al (Qu{\'e}bec) \\ H3C 3P8 \\ Canada}
\email{apostolov.vestislav@uqam.ca}

\author[J. Keller]{Julien Keller} \address{Julien Keller \\ Aix Marseille Universit\'e, CNRS, Centrale Marseille, Institut de Math\'ematiques
de Marseille, UMR 7373, 13453 Marseille, France}
 \email{julien.keller@univ-amu.fr}

\begin{abstract} Let $\PP(E)$ be the projectivization of a holomorphic vector bundle $E$
over a compact complex curve $C$. We characterize the existence of an extremal K\"ahler metric on $\PP(E)$ in terms of relative K-polystability and the fact that $E$ decomposes as a direct sum of stable bundles.
\end{abstract}

\maketitle

\tableofcontents

\section{Introduction}

Let $M=\PP(E)$
be the complex manifold underlying the total space of the projectivization of a holomorphic vector bundle $E \to C$
over a compact complex curve $C$. In this paper, we are interested to understand when  $\PP(E)$ admits an extremal K\"ahler metric in the sense of Calabi~\cite{cal-one}, 
and if such a special metric does exist, which K\"ahler classes of $\PP(E)$ admit extremal K\"ahler metrics. 
A K\"ahler class $\Omega$ endowed with a K\"ahler metric is refereed to as an {\it extremal class}.
\par  By the openness of the extremal K\"ahler classes on $M$ (see \cite{Le-Sim, Fuj-Sch}) and using the fact that $H^{2,0}(M, \mathbb C)=0$, without loss of generality we can restrict our study to  rational classes. Furthermore,  as extremal  K\"ahler classes are invariant under a positive rescaling, we can even only consider integral classes, i.e. $\Omega= 2\pi c_1(\cL)$ for a positive line bundle $\cL$ on $M$. Such bundles on $M=\PP(E)\to C$ are of the form 
$$\cL=\cL_{q,p} = \cO(q)_{\PP(E)}\otimes \cO(p)_C, \, q, p \in \Z, q>0,$$  where $\cO(1)_C$ denotes (the pull-back to $M$) of {\it any} holomorphic line bundle over $C$ of degree $1$.  We note that $\cL=\cL_{q,p}$ becomes positive for $p/q \gg 0$ and thus defines a  {\it polarization} on $M$.

\par Before discussing the case of extremal metrics, let us recall some results about the particular case of constant scalar curvature K\"ahler metrics (CSC K\"ahler for short).  In the case of $M=\PP(E)\to C$, the existence of CSC K\"ahler metrics and its link to K-polystability in the sense of \cite{Do2,Tian} are completely settled thanks to the works \cite{ACGT,RT}.
\begin{thm}\label{thm:csc}\cite{ACGT,RT} Let $E$ a holomorphic vector bundle and let  $M=\PP(E) \to C$ be its projectivisation.  The following three conditions  are equivalent:
\begin{enumerate}
\item[\rm (i)] $M$ admits a CSC K\"ahler metric in any class $2\pi c_1(\cL)$;
\item[\rm (ii)] $M$ is K-polystable for any polarization $\cL$;
\item[\rm (iii)] $E$ is polystable, i.e. decomposes as the sum of stable bundles of same slopes;
\end{enumerate}
\end{thm}
\begin{rem} The notion of stability for bundles refers here to the classical notion of Mumford--Takemoto stability. The equivalence ${\rm (iii)} \Longleftrightarrow {\rm (i)}$ is established in \cite[Theorem~1]{ACGT}  when the base $C$ has genus  ${\bf g} \ge 2$,  and in \cite[Theorem~5.13]{RT} when ${\bf g} \ge 1$ (by using the result of \cite{stoppa}).  This equivalence also holds true for ${\bf g} =0$ as a consequence of the Lichnerowicz--Matsushima theorem, by noting that in this case $E$ splits as a direct sum of line bundles over $C=\C \PP^1$.  The equivalence ${\rm (ii)} \Longleftrightarrow {\rm (iii)}$  follows by  \cite[Theorem.~5.13]{RT},  by noting again that the case ${\bf g}=0$ can be treated apart by observing that the  usual  Futaki invariant associated to $2\pi c_1(\cL)$ vanishes if and only if $E$ is the sum of line bundles  over $\C \PP^1$ of the same degree. 
\end{rem}

In \cite{ACGT}, it was introduced the following conjecture in view of the classification of projective bundles over a curve,  admitting extremal metrics.
\begin{conj}\label{c:extremal}\cite{ACGT} Let $E$ a holomorphic vector bundle and let  $M=\PP(E) \to C$ be its projectivisation. The following three conditions  are equivalent:
\begin{enumerate}
 \item[{\rm (a)}] $M$ admits an extremal K\"ahler metric;
 \item[{\rm (b)}] $M$ is relative K-polystable for a certain polarization $\cL$;
 \item[{\rm (c)}] $E$ decomposes as a direct sum of stable bundles.
\end{enumerate}
\end{conj}

We refer to \cite{gabor} for the notion of relative K-polystability. Our main result is the following theorem. 
\begin{thm}\label{thm:main}
 Conjecture \ref{c:extremal} is true.
\end{thm}

We do now some comments. Firstly, Conjecture  \ref{c:extremal} is almost optimal in view of the existence problem of extremal K\"ahler metric. Actually in the light of Theorem \ref{thm:csc}, it is natural to ask if Conjecture \ref{c:extremal} could be completed by 
\begin{itemize}
 \item[(d)] $M$ admits an extremal K\"ahler metric in any K\"ahler class;
 \item[(e)] $M$ is relative K-polystable polarization for any polarization $\cL$;
\end{itemize}
But this does not hold.  In general, with $E$ direct sum of stable bundles,  $M$ may also admit polarizations which  are not relatively K-polystable, nor extremal, see e.g. \cite[Proposition 5]{ACGT} or \cite[Theorem 6]{ACGT-0}. To strengthen  Conjecture  \ref{c:extremal}, it would be natural to ask that conditions (a) and (b) occur precisely for the {\it same} classes. This is precisely the Yau--Tian--Donaldson conjecture extended to the setting of extremal K\"ahler metrics by Sz\'ekelyhidi~\cite{gabor}.
\begin{conj}[Relative Yau-Tian-Donaldson conjecture]\label{c:YTD} For any polarization $\cL$ on $M=\PP(E)\to C$, the following two conditions are equivalent
\begin{enumerate}
\item[\rm (a')] $2\pi c_1(\cL)$ is extremal;
\item[\rm (b')] $(M, \cL)$ is relatively K-polystable. 
\end{enumerate}
\end{conj}
By virtue of \cite{gabor-stoppa} (see also \cite[Theorem~1.2]{dervan-2}), we know that (a') implies (b')  on any polarized variety. In direction of this conjecture, Theorem~\ref{thm:main}  combined with some previous results allows us to establish the following

\begin{cor}\label{cor1}
 Let $E\to C$ a holomorphic vector bundle over a complex curve $C$ and write  $E=U_1\oplus \cdots \oplus U_s$ as a direct sum of indecomposable sub-bundles.
 The relative Yau--Tian--Donaldson conjecture is true for a polarization $\cL_{q,p}$ on $\PP(E)$ in the following cases:
 \begin{enumerate}
          \item $s=1$ or $s=2$;  
     \item $s\geq 3$ and $E$ is polystable;
     \item  $s\geq 3$ and one of the $U_i$ is unstable;
     \item  $s\geq 3$ and $p/q$ is large enough.
        \end{enumerate}
\end{cor}

\begin{rem}
An example from \cite{ACGT-0} strongly suggests that for reaching the remaining cases ($s\geq 3$, $U_i$ stables of different slopes and $p/q$ not large),  one would  need  to  enforce  the notion of relative K-polystability of $(M, \cL)$. This would require to consider test configurations  with ``irrational''  line bundles (i.e. formal tensor powers of line bundles with real coefficients).  There are two current approaches to this.  The first is the notion of {\it K\"ahler} relative K-stability, which originates in \cite{RT} and was recently  developed in \cite{SD2016,DR2016,dervan-2}. The other one  is  the  notion of {\it uniform} relative K-stability, as introduced in \cite{Sz2015, BHJ, dervan-1}. 
\end{rem}

\smallskip

Eventually, one would expect that some of the results discussed above can be extended to projective holomorphic vector bundles $M=\PP(E) \to B$ over a base $(B, L_B)$ which itself is a polarized variety admitting an extremal K\"ahler metric in $2\pi c_1(L_B)$. This is evidenced in the works \cite{hong, Bro,LSeyy,KR1}.

\smallskip

We sum up now the general structure of the paper. In Section \ref{Preliminaries} we present the required material about (relative) Donaldson--Futaki invariant. In Section \ref{mainsect}, we construct a test-configuration and compute by two different ways the associated relative Donaldson--Futaki invariant (Sections \ref{s:calabi-anzats}, \ref{computgc} and \ref{sectDF}). Our  first approach is based on differential-geometric ingredients from \cite{ACGT} and has the advantage to apply to any K\"ahler
class (rational or not), thus evidencing the K\"ahler feature of the K-stability in line with the recent work \cite{dervan-2}. The second approach is algebro-geometric, following the original arguments in \cite{RT}, and has the merit  to cover the case when the genus of $C$ equals 1. The proof of our main result is then given at the end of Section \ref{generalcase}.  The proof of Corollary \ref{cor1} is obtained in Section \ref{sectcor}. The appendix (Section \ref{appendix}) contains certain technical results.

\section{Preliminaries}\label{Preliminaries}
\subsection{The Donaldson--Futaki invariant} 
Suppose $(M, \cL)$ is a polarized variety endowed with a $\C^{\times}$ action $\rho$ with a lift to $\cL$.  Let $A^{\rho}_k$ be the infinitesimal generator of the induced linear action on the vector space $H_k:= H^0(M, \cL^k)$ of holomorphic sections of $\cL^{k}$, and denote by
$$d_k := {\rm dim} \  H_k, \  \ w_k(\rho):= \mathrm{Tr}\,A^{\rho}_{k}.$$
It turns out (by using Riemann--Roch) that for $k\gg 0$, $d_k$ and $w_k$  are polynomials 
$$d_k = a_0 k^{n} +  b_{0} k^{n-1} + O(k^{n-2}), \ \  w_k(\rho) = a(\rho)  k^{n+1} + b(\rho) k^n + O(k^{n-1}).$$
We then define 
\begin{defn} The {\it algebraic Donaldson--Futaki invariant} of $(M, \cL, \rho)$  is  
$$\mathfrak{F}^{alg}(\rho):= \frac{a_0b(\rho)- a(\rho)b_0}{a_0}.$$
\end{defn}
We shall use this definition in the case when $(M, \cL)$ is a smooth polarized variety. We notice that there are different sign choices in the literature  for the infinitesimal generator of the induced linear action on $H_k$,  thus introducing a sign difference in the definition of the (algebraic) Donaldson--Futaki invariant, see e.g. \cite[p. 141]{gabor-book}. We shall use in this paper the following convention,   which  agrees with  \cite{gauduchon-book} and,  up to a positive constant,  with \cite[(7.14)]{gabor-book}.
\begin{defn}\label{d:convention} Let $(M, \cL)$ be a smooth polarized variety endowed with a $\C^{\times}$ action $\rho$ with a lift to $\cL$,  denoted by $\hat \rho$.  We  let $e^{\sqrt{-1}t}, t \in {\R}$ be the circle subgroup of $\C^{\times}$. Then, the {\it infinitesimal generator} $A^{\rho}$ for the action of $\hat \rho$ on the space of smooth sections $\Gamma(\cL)$ is defined to be
$$(A^{\rho} (s))(x) := \sqrt{-1} \frac{d}{dt}_{|_{t=0}} \Big(\hat \rho(e^{\sqrt{-1}t})\big(s(\rho(e^{-\sqrt{-1}t})(x))\big) \Big), \ \ s \in \Gamma(\cL), x\in M.$$
\end{defn}
It is shown in \cite{Do2} 
(see also \cite{gauduchon-book, gabor-book}) that  the above definition  agrees, up to a  factor of $\frac{1}{4(2\pi)^n}$,  with the differential geometric definition of the Futaki invariant \cite{futaki}, i.e. 
$$4 (2\pi)^n \mathfrak{F}^{alg}(\rho) = \mathfrak{F}_{\Omega}(K_{\rho}) = \int_{M}{\rm Scal_g} f_{\rho} v_g, $$
where $n= {\rm dim}_{\C} M$, $K_{\rho}$ is the (real) holomorphic vector field on $M$ induced by the action of  $S^1$ via $\rho$,  $\Omega= 2\pi c_1(\cL)$ is the K\"ahler class determined by $\cL$, $g$ is any $S^1$-invariant K\"ahler metric in $\Omega$, ${\rm Scal}_g$ is its scalar curvature, and $f_{\rho}$ denotes the Killing potential  of mean value zero for $K_{\rho}$  with respect to $g$. 

\subsection{Test configurations and K-polystability} 
Recall the following definitions from \cite{Tian2} and \cite{Do2}.
\begin{defn}\label{d:test-configuration} Let $(M,\cL)$ be a normal polarized variety. A {\it test configuration} for $(M, \cL)$ is a normal variety $\mathcal{M}$ endowed with a line bundle $\mathcal P$ together with 
\begin{enumerate}
\item[\rm (i)] a $\C^{\times}$ action $\rho$ on $\mathcal M$ with a lift to $\cP$;
\item[\rm (ii)] a  $\C^{\times}$ equivariant map $\pi_{\C}: \mathcal M \to \C$ where $\C^{\times}$ acts on a standard way on $\C$,
\end{enumerate}
such that $\pi_{\C} : \mathcal M \to \C$ is a flat family  with  $\mathcal{P}$  being relatively ample and,  for any $t\neq 0$,  the fibre $(M_t, {\mathcal P}_{M_t})$ of $\pi_{\C}$ is isomorphic to $(M, \cL^r)$ for some fixed $r\in \N$. The number $r$ is called  {\it exponent} of the test configuration.

A test configuration is said to be a {\it product configuration} if $\cM=M\times
\C$ and $\rho$ is given by a $\C\mult$ action on $M$ (and scalar
multiplication on $\C$).
\end{defn}
Notice that  for any test configuration $(\cM, \cP, \rho)$ for $(M, \cL)$,  $\rho$  induces a $\C^{\times}$ action on the central fibre $(M_0, \cL_0)$ (which we still denote by $\rho$). With our convention in Definition~\ref{d:convention}, we then  have 
\begin{defn}\cite{Do2}
The {\it Donaldson--Futaki invariant  of the test configuration} $(\mathcal{M}, \cP, \rho)$ for $(M, \cL)$  is  the Donaldson--Futaki invariant $\mathfrak{F}^{alg}(\rho)$ of its central fibre $(M_0, \cL_0)$. The variety $(M,\cL)$ is said to be K-{\it polystable}  (resp. K-stable) if the Donaldson--Futaki
invariant of any normal  test configuration for $(M, \cL)$ is non-negative, and equal to zero if and
only if the test configuration is a product configuration (resp a trivial test configuration).
\end{defn}
This implies in particular that the Donaldson--Futaki invariant of any $\C^{\times}$ action on $(M, \cL)$  must be zero, so the notion is adapted to the study of cscK (in particular K\"ahler--Einstein)  metrics. 

\subsection{Relative K-polystability}\label{sectRK} In order to account for the obstructions related to the extremal K\"ahler metrics, G. Sz\'ekelyhidi has introduced relative version of the above notions  as follows.

Suppose $(M, \cL)$ is a polarized variety endowed with two commuting  $\C^{\times}$ actions $\rho_1$ and $\rho_2$.  We first define an inner product $\langle \rho_1, \rho_2 \rangle$ for such actions. For that, we take lifts of $\rho_1$ and $\rho_2$  to $\cL$ and  consider the  infinitesimal generators $A^{\rho_1}_k$ and $A^{\rho_2}_k$ of the actions on
$H_{k}$. Then for $k$ sufficiently large, 
$$\mathrm{Tr}\,(A^{\rho_1}_k A^{\rho}_2)= a(\rho_1,\rho_2)k^{n+2} + O(k^{n+1})$$ is a
polynomial of degree at most $n+2$ and we let
\begin{defn} The {\it inner product} of two commuting $\C^{\times}$ actions $\rho_1$ and $\rho_2$ on $(M, \cL)$ is defined by
$$\langle \rho_1, \rho_2 \rangle := a(\rho_1,\rho_2) - \frac{a(\rho_1) a(\rho_2)}{a_0}.$$
\end{defn}
Notice that $\langle \rho_1, \rho_2 \rangle$ is the leading coefficient of the expansion in $k$ of $\mathrm{Tr}\,(\mathring{A}^{\rho_1}_k \mathring{A}_k^{\rho_2})$ of the traceless parts of $A^{\rho_i}_k$, so it is independent of the choice of liftings. It is shown in \cite{gabor} that when $M$ is smooth, the above definition agrees  up to a factor of $1/(2\pi)^n$, with the Futaki--Mabuchi bilinear form on Killing potentials, i.e. if for any K\"ahler metric  $g$ in $\Omega= 2\pi c_1(\cL)$ which is invariant under the $S^1$ actions corresponding to $\rho_1$ and $\rho_2$ we denote by $f_{\rho_1}$ and $f_{\rho_2}$ the Killing potentials of zero mean with respect to $g$,   corresponding the induced Killing vector fields, then
$$\langle \rho_1, \rho_2 \rangle = \frac{1}{(2\pi)^n} \int_M f_{\rho_1} f_{\rho_2} v_g.$$
We shall next fix a maximal torus $\T^{\ell}$ in the automorphism group ${\rm Aut}(M,\cL)$ and denote by $\rho_1, \ldots, \rho_{\ell}$ the $\C^{\times}$ corresponding to the $S^1$ generators of $\T^{\ell}$.
\begin{defn} The {\it extremal} $\C^{\times}$ action  $\rho_{\rm ex}$ of $(M, \cL, \T^{\ell})$ is a $\C^{\times}$ subgroup of the complexification $\T^{\ell}_c$ defined by the system of $\ell$ linear conditions
$$\langle \rho_{\rm ex}, \rho_i \rangle = \mathfrak{F}^{alg}(\rho_i).$$
\end{defn}

\begin{defn} Let $\rho_0$ be a distinguished $\C^{\times}$ action on the polarized manifold $(M, \cL)$. The {\it $\rho_0$-relative  Donaldson--Futaki invariant}  of $(M, \cL, \rho_0)$ is defined for any $\C^{\times}$ action $\rho$ commuting with $\rho_0$ by
\begin{equation*}
{\mathfrak F}^{alg}_{\rho_0}(\rho) := {\mathfrak F}^{alg} (\rho) - \frac{\langle \rho_0, \rho \rangle}{\langle \rho_0, \rho_0\rangle}{\mathfrak F}^{alg}(\rho_0).
\end{equation*}
\end{defn}

We now apply the above to  a test configuration.

\begin{defn} \label{d:relative-futaki-gabor} \cite{gabor,stoppa-corrigendum}
A test configuration $(\cM,\cP, \rho)$ for $(M,\cL)$ is {\it compatible} with  a  fixed maximal torus $\T^{\ell} \subset {\rm Aut}(M, \cL)$  if
there is a $\T^{\ell}$ action on $(\cM, \cP)$, commuting with $\rho$ and preserving 
$\pi_{\C} \colon \cM\to\C$, which induces  the trivial action on $\C$,  and 
restricted to $(M_t,\cP\restr{M_t})$ for $t\neq 0$ coincides with the original $\T^{\ell}$ action  under the isomorphism with $(M,L)$ via $\rho$.

In this case,  we have an induced action of $\T^{\ell}$ on the central fibre $M_0$, and we denote by $\rho_{\rm ex}^M$ the $\C^{\times}$ action on $M_0$ corresponding to the extremal $\C^{\times}$ action $\rho_{\rm ex}$ on $(M, \cL, \T^{\ell})$. 

Now, the {\it relative Donaldson--Futaki invariant} of  a $\T^{\ell}$ compatible test configuration  for $(M,\cL, \T^{\ell}))$  is defined to be the  $\rho^M_{\rm ex}$-relative Donaldson--Futaki invariant of $(M_0, \cL_0)$ of the induced $\C^{\times}$ action $\rho$.

A polarized variety manifold $(M,\cL)$    is
relatively K-{\it polystable} (resp. K-stable)  with respect to a maximal torus $\T^{\ell} \subset {\rm Aut}(M, \cL)$ if the relative Donaldson--Futaki invariant
of any normal  test configuration $(\cM,\cP, \rho)$ for $(M, \cL)$ compatible with $\T^{\ell}$ is non-negative,
and equal to zero if and only if $(\cM,\cP)$ is a product configuration (resp a trivial configuration). 
\end{defn}

\begin{rem}\label{r:new-Futaki} More recently, following the works of Wang \cite{Wang4} and Odaka~\cite{Odaka1, Odaka2},   a topological interpretation  of the Donaldson--Futaki invariant was given in terms of an integration over the total space of a given test configuration. Among other applications, this point of view  led to the definition of the stronger notion of  {\it K\"ahler} (relative) K-polystability in \cite{SD2016,DR2016, dervan-2},  where one also takes in consideration the sign of the (relative) Donaldson--Futaki  over ``irrational''  polarizations $\mathcal L$ of $M$. We shall not use this point of view explicitly in this paper.  However, the Reader could notice that our differential-geometric approach to the computation of the relative Donaldson--Futaki invariant is well-adapted to deal with the K\"ahler relative K-polystability in the sense of \cite{dervan-2}.
\end{rem}

\section{Proof of Theorem~\ref{thm:main} and Corollary \ref{cor1}}\label{mainsect}

One direction of  Theorem~\ref{thm:main}, namely $(c)\Longrightarrow (a)$, follows from the facts that an extremal  K\"ahler metric exists  in any  polarizations $\cL=\cL_{q,p}$ with $p/q \gg 0$ (see \cite[Theorem 3]{ACGT}  or \cite{Bro}). Moreover, $(a) \Longrightarrow (b)$  is a consequence of the general result of \cite{gabor-stoppa}, i.e the existence of an extremal K\"ahler metric in $2\pi c_1(\cL)$ implies that $(M, \cL)$ is relative K-polystable. We shall thus focus on establishing $(b)\Longrightarrow (c)$. As any vector bundle $E$ over $\C \PP^1$ decomposes as the direct  sum of line bundles (which are automatically stable),  we shall also assume  from now on  that the base $C$ has genus ${\bf g} \ge 1$.   
\begin{assumption} $C$ is a compact complex curve of genus ${\bf g} \geq 1$.
\end{assumption}

As  the cohomology $H^2(M, \R)$  of $M=\PP(E)$  is $2$-dimensional, up to rescaling, the K\"ahler cone of $M$ is $1$-dimensional. Similarly, it is well-known  that any holomorphic line bundle $\cL$ on $M$ can be written as  
$$\cL= \cL_{q,p} := \cO(q)_{\PP(E)}\otimes \cO(p)_C, \ p,q \in \Z,$$
where, as usual,  $\cO(1)_{\PP(E)}$ denotes the anti-tautological line bundle of $E$ (defined on $\PP(E)$) and $\cO(1)_C$ stands for  (the pull back to $M$ of)  {\it any} degree $1$ holomorphic line bundle over $C$, see for instance \cite[Section 3]{Miyaoka} for details. If $\Omega = 2\pi c_1(\cL_{q,p})$ is  a K\"ahler class, evaluation over the fibre of $\PP(E)$ shows that $q>0$, thus any polarization on $M=\PP(E)$ can be written  as $\cL=\cL_{q, p}$ with $q>0$ (notice  that $\cL_{q,p}$ becomes positive when $p/q \gg 0$).  Clearly, both properties of existence of extremal K\"ahler metric and relative K-polystability of the polarization $\cL$ on $M$ are invariant under taking tensor powers $\cL^{\otimes k}=\cL^{k}$. As it will turn out in our specific situation, the same phenomena happens under changing the polarization $\cO(1)_C$ of the base curve $C$.  It will be useful to normalize the choice of  such polarizations, by introducing the following
\begin{notation} \label{n:polarization} We let $\cL_m := \cO(1)_{\PP(E)}\otimes \cO(m)_C, m \in \Q$ denote the class of holomorphic line bundles $\cL_{q,p}$ over $M$ such that $q>0$ and $p/q=m$.   
\end{notation}
In all of the arguments below  involving $\cL_m$  one can take some (and hence any) line bundle $\cL_{q,p}$ as above.

\smallskip We denote by ${\rm Aut}^{red}(M)$ the reduced automorphism group of $M= \PP(E)\to C$ (see e.g. \cite{gauduchon-book}) whose Lie algebra $\mathfrak{h}^{red}(M)$ consists of all holomorphic vector fields with zero on $M$, and let ${\rm Aut}^{red}_C(M)$ be the subgroup of ${\rm Aut}^{red}(M)$ of elements  which preserve $C$ (i.e. act on each fibre),  with Lie algebra $\mathfrak{h}^{red}_C(M)$. As $M$ is a locally trivial  holomorphic $\C \PP^{n-1}$-fibration over $C$, we have an exact sequence of Lie algebras
\begin{equation}\label{exact-sequence}
0 \to \mathfrak{h}^{red}_C(M) \to \mathfrak{h}^{red}(M) \to \mathfrak{h}^{red}(C) \to 0,
\end{equation}
where $\mathfrak{h}^{red}(C)$ its the Lie algebra of holomorphic vector fields with zeroes on $C$. Under the assumption ${\bf g}(C) \ge 1$, we have $\mathfrak{h}^{red}(C)=0$, so that $\mathfrak{h}^{red}(M)=\mathfrak{h}_C^{red}(M)$. We let $(\ell-1)$ with $\ell \ge 1$ denote the rank of ${\rm Aut}^{red}(M)$ (which is also the rank of ${\rm Aut}^{red}_C(M)$ by the preceding).  Thus, $\ell$ equals the number of  summands  in the decomposition 
\begin{equation}\label{decomposition}
E= \bigoplus_{k=0}^{\ell-1} U_k.
\end{equation}
of $E$ as direct sum of indecomposable holomorphic sub-bundles $U_k$.  We want to show that, in general, each $U_k$ is stable when $M$ is relative K-polystable with respect to the polarization  ${\cL}_{m}$ of $M$. Without loss,  we deal with $U_{0}$ and  assume $\rk(U_0)>1$.

\subsection{Constructing a test configuration} This construction follows \cite[Remark 5.14]{RT} and \cite[Section 3]{rollin}. For each strict sub-bundle $L\subset U_{0} \subset E$,  we consider the exact sequences of holomorphic vector bundles
\begin{equation}\label{extension}
\begin{split}
0 &\to L \to U_0 \to F_0 \to 0, \\
0 &\to L \to E \ \to F   \ \to 0,
\end{split}
\end{equation}
where $F_0 = U_0/L$ and $F= E/L \cong F_0 \oplus \big(\bigoplus_{k=1}^{\ell-1}U_k\big)$. Thus, $E$ is given by an element $e\in {\rm ext}^1(L,F)$, coming from an element  (still denoted by $e$) of ${\rm ext}^1(L,F_0)$; as  $U_0$ is indecomposable,  $e\neq 0$, and one can consider the smooth family
$\mathcal M := (M_t, t), t \in \C,$ where $M_t:=\PP(E_t)$ and $E_t$ is the extension of $(L,F)$ corresponding to $t e \in {\rm ext}^1(L,F_0)=H^1(M, L\otimes F_0^*)$ for $t\in \C$. As  explained in \cite[Section 3.1]{rollin}, $\mathcal  M = \PP(\mathcal{E}) \to C \times \C$ is itself a complex ruled manifold, where  $\mathcal{E}$ is a holomorphic vector bundle  whose restriction to $C\times \{t\}$ is $E_t$. We denote by $\pi_{\C} : \mathcal M \to \C$ the natural holomorphic projection on the  $\C$-factor.  As $E_t \cong E$ for $t\neq 0$, we have that   $\pi^{-1}(t)= M_t \cong M$, whereas  
$$\pi_{\C}^{-1}(0)=M_0= \PP\Big(\bigoplus_{k=0}^{\ell} V_k \Big) \to C,$$
where  we have set $V_0:=U_0/L,\,\, V_1:=L, \,\,  V_k := U_{k-1}, \, k=2, \ldots, \ell.$  As shown in \cite[Lemma~3.1.1]{rollin},  there is a natural  $\C^{\times}$ action $\rho_L$ on $\mathcal M$,  making $\pi_{\C}$ equivariant with respect to the standard action on $\C^{\times}$,  and which induces a ${\mathbb C}^{\times}$ action (still denoted by $\rho_L$) on the central fibre $M_0$,   given by  the fibre-wise multiplication with $\lambda \in \C^{\times}$ on the factor $V_1=L$  in the decomposition 
$$V := \bigoplus_{k=0}^{\ell} V_k.$$
Given a polarization $\cL_{m}=\cO(1)_{\PP(E)}\otimes\cO(m)_C$ on $M$ (we can work with any line bundle representing $\cL_m$, see Notation~\ref{n:polarization}), consider on  $\mathcal M$ the  (class of rational) holomorphic line bundles $\mathcal P_{m} := \cO(1)_{\PP(\cE)} \otimes \cO(r)_C$  which restricts to $\cL_{m}^t=\cO(1)_{\PP(E_t)} \otimes \cO(m)_C$ on each  fibre $M_t =\pi^{-1}_{\C}(t)$.  The $\C^{\times}$ action $\rho_L$ on $\mathcal M= \PP(\mathcal E)$ comes from an action preserving the vector bundle $\cE \to C\times \C$ (and acting trivially on $C$), so $\rho_L$ naturally lifts to an action on $\mathcal{P}_{m} \to \mathcal M \to \C$. It thus follows that for any $t\neq 0$, $(M_t, \cL^t_{m})$ is  a polarized variety isomorphic to $(M, \cL_{m})$. Furthermore, the holomorphic line bundle $\cL^0_{m}$ induced on the central fibre $M_0$ must be at least semi-ample.  As  the condition for $\cL_{m}$ to be ample on $M$ is relatively open  with respect to  $m \in \Q$, it follows that $\cL^0_{m}$ must be ample too. We thus conclude that  $(\mathcal M, \rho_L, {\mathcal P}_{m})$ defines  a test-configuration for $(M, \cL_{m})$  (the flatness of the morphism $\pi_{\C} : (\mathcal M, \mathcal P_{m}) \to \C$ is a direct consequence of the surjectivity of $\pi_{\C}$ and the fact that the central fibre is smooth).   We finally  notice that the rank of the reduced automorphism group of the central fibre $M_0$ is at least  $\ell$,   whereas the rank of the same group on $M_t$   is $(\ell-1)$ for $t \neq 0$, showing that the test configuration $\mathcal M$ is normal and not a product configuration \cite{Li-Xu, stoppa-corrigendum}.  We thus have established the following
\begin{lemma}\label{l:test-configuration} Given a completely decomposable vector bundle $E=\bigoplus_{k=0}^{\ell-1} U_k \to C$, a polarization $\cL_{m}= \cO(1)_{\PP(E)} \otimes \cO(m)_C$ on $M= \PP(E)$  and  a sub-bundle  $L \subset U_0$,   the data $$(\pi_{\C}: \mathcal M \to \C,  \rho_L, \mathcal P_{m})$$  define  a normal test configuration for $(M=\PP(E), \cL_{m})$ which is not a product configuration and with central fibre $(M_0=\PP(V), \cL^0_{m}),$ where
$$V = \bigoplus_{k=0}^{\ell} V_k , \,   \ V_0 : =U_0/L, \ , V_1: =L,  \, V_k := U_{k-1}, \, k=2, \ldots, \ell.$$
The induced $\C^{\times}$ action $\rho_L$ on $M_0$ is given by 
$$\rho_L(\lambda) \cdot (x, [e_0, e_1, \cdots, e_{\ell}]) = (x, [e_0, \lambda e_1, \ldots, e_{\ell}]), $$
where $x$ is a point on $C$ and $(v_0, v_1, \ldots, v_{\ell})$ is a vector in the fibre $V_x$  of $V$.
\end{lemma}

\subsection{Relative Donaldson--Futaki invariant}\label{secRFI}
The central fibre $M_0$ is a smooth complex variety, endowed with a holomorphic action of  the torus $\T^{\ell}$, coming from the diagonal action of $\T^{\ell+1}$ on $V= \bigoplus_{k=0}^{\ell} V_k$.  We choose any K\"ahler metric $g$ on $M_0$ in the K\"ahler class $\Omega=2\pi c_1(\cL_m^0)$, which is invariant under the action of $\T^{\ell}$.  The action of the sub-torus $\T^{\ell-1} \subset \T^{\ell}$ by diagonal multiplications on the factors $V_2, \ldots, V_{\ell}$ extends to each  fibre $M_t, t \neq 0$, and on $\mathcal M$.  As $\T^{\ell-1}$ is a maximal torus in the connected component of identity of ${\rm Aut}^{red}(M_t)$ for any $t\neq 0$, it follows that the extremal vector field $K_{\rm ex}$ of $(M_t, \Omega)$ belongs to ${\rm Lie}(\T^{\ell-1}) \subset {\rm Lie}(\T^{\ell})$ and is independent of $t$ (as $M_t \cong M_{t'}$ via $\rho_L$ and the action of $\T^{\ell-1}$ on $\mathcal M$ commutes  with $\rho_L$). We shall denote this vector field by $K_{\rm ex}^M$ and let $f^M_{{\rm ex}}$ be the Killing potential of $K^M_{\rm ex}$ of zero mean value with respect to $g$. As the central fibre $M_0$ is a smooth variety, the relative Donaldson--Futaki invariant  is  computed up to a positive normalization factor  by the differential-geometric quantity (see \cite{gabor} or Section~\ref{Preliminaries})
\begin{equation}\label{relative-futaki-2}
\begin{split}
\mathfrak{F}_{\rho_{\rm ex}}(\rho_{L}) &= \mathfrak{F}_{\Omega}(K_{L}) - \frac{\int_{M_0} (f_{L} f^M_{{\rm ex}})  v_g}{\int_{M_0} (f_{{\rm ex}}^M)^2 v_g} \mathfrak{F}_{\Omega}(K^M_{\rm ex})\\
                                                   &= \int_{M_0} {\rm Scal}_g f_{L} v_g -  \frac{\int_{M_0} (f_{L} f^M_{{\rm ex}})  v_g}{\int_{M_0}(f_{{\rm ex}}^M)^2 v_g} \int_{M_0} {\rm Scal}_g f^M_{{\rm ex}} v_g,
\end{split}
\end{equation}
where $K_{L}$ denotes the generator for the induced $S^1$ action by $\rho_{L}$ (again a subgroup of $\T^{\ell}$),  $f_{L}$ is its Killing potential of zero mean value with respect to  $g$. Of course, the r.h.s. of \eqref{relative-futaki-2} is independent of the choice of $\T^{\ell}$-invariant K\"ahler metric $g$ in $\Omega$.

As explained in the proof of Lemma 3 in  \cite{ACGT},   one can extend the $\T^{\ell}$ invariant K\"ahler metric $(g, \omega)$ on $M_0=(M, J_0)$ to a smooth family of $\T^{\ell-1}$ invariant K\"ahler metrics $(g_t, \omega_t)$ on $M_t=(M, J_t)$ (at least for $|t|<\varepsilon$) and then use the equivariant Moser lemma in order to find a $\T^{\ell-1}$ equivariant  family of diffeomorphisms $\Phi_t$ on $M$,  which send the complex structure $J_t$ of $M\cong M_t$ to a complex structure $\tilde J_t$ on $M$,  compatible  with the initial symplectic form $\omega$. As $\Phi_t$ commutes with the action of $\T^{\ell-1}$ and $K^M_{\rm ex} \in {\rm Lie}(\T^{\ell-1})$,  $\Phi_t$ preserves $K^M_{\rm ex}$.  In this symplectic setting, it is shown in \cite[Lemma~2]{ACGT} (see also \cite{lejmi}) that $f^M_{{\rm ex}}$ can be obtained as the $L^2$-projection of the scalar curvature of {\it any} $\T^{\ell-1}$ invariant K\"ahler metric compatible with $\omega$ to the finite dimensional space of normalized  hamiltonians  for  the $\T^{\ell-1}$ action on $(M, \omega)$. In particular, with respect to the initial metric  $g$, we have that $f^M_{\rm ex}$ coincides with  the $L^2$-projection ${\rm Scal}_g^{\T^{{\ell}-1}}$ of ${\rm Scal}_g$ to the space of normalized hamiltonians of $\T^{\ell-1} \subset \T^{\ell}$. In particular, we have  $$ \int_{M_0} {\rm Scal}_g f^M_{\rm ex} v_g =\int_{M_0}(f_{\rm ex}^M)^2 v_g,$$ so that \eqref{relative-futaki-2} becomes
\begin{equation}\label{relative-futaki-3}
\mathfrak{F}_{\rho_{\rm ex}}(\rho_{L})= \int_{M_0} {\rm Scal}_g f_{L} v_g - \int_{M_0} (f_{L} {\rm Scal}_g^{\T^{\ell -1}})  v_g.
\end{equation}
From this point of view, \eqref{relative-futaki-3} can be entirely computed from the symplectic structure $\omega$ on $M_0$, endowed with the hamiltonian action of  $\T^{\ell}$.  We thus have
\begin{lemma}\label{l:futaki-symplectic} The r.h.s of \eqref{relative-futaki-3} does not depend on the choice of an $\omega$ compatible, $\T^{\ell}$ invariant  K\"ahler metric  $(g, J_g)$ on $(M, \omega)$ nor on the choice  of a  $\T^{\ell}$ invariant K\"ahler metric $\tilde g$ on $(M, J_g)$ within the K\"ahler class $\Omega=[\omega]$.
\end{lemma}

 \subsection{Generalized Calabi Ansatz}\label{s:calabi-anzats}  When $C$  is of genus ${\bf g} \ge 2$, by using Lemma~\ref{l:futaki-symplectic} and the Narasimhan--Ramanan approximation theorem~\cite{NR},  we can compute \eqref{relative-futaki-3} with respect to an $\omega$ compatible,  $\T^{\ell}$ invariant complex structure $\tilde J$ on $M_0$, corresponding to taking {\it stable} holomorphic structures on each $V_k$, see \cite[Lemma~2]{ACGT}. Furthermore, in this case, we can use  the generalized Calabi Ansatz of \cite{ACGT} in order to choose a particularly simple metric $g_c$ in the class $\Omega=[\omega]$ on $(M, \tilde J)$, which will make the computation of  \eqref{relative-futaki-3} explicit. 
 
To simplify the notation, we shall assume throughout this section that $M_0=(M, J_0)$ is a ruled complex manifold 
 $$M_0 = \PP\Big(\bigoplus_{k=0}^{\ell} V_k \Big) \to C, $$
over a compact complex curve of genus ${\bf g} \ge 2$,  and $V_k$ are stable vector bundle over $C$. This is a special case of the semi-simple rigid  toric fibre-bundles considered in \cite{ACGT}, see Sect. 2.2 {\it loc cit}.  

We  introduce a family of K\"ahler metrics $(g_c, \omega_c)$ on $M_0$, parametrized by a real constant $c$,  as follows: As each $V_i$ is  a stable and therefore projectively-flat bundle over $C$,  it admits a projectively-flat hermitian metric $h_i$ whose Chern curvature is  $\mu(V_i){\rm Id}\otimes \omega_C$, where the topological constant
$$\mu(V_i):=  \frac{\deg(V_i)}{\rk(V_i)}= \int_C c_1(V_i)/{\rm rk}(V_i)$$
is the {\it slope} of $V_k$,  and $\omega_C$ is the K\"ahler form of the metric $g_C$ on $C$  of  constant scalar curvature $2(1-{\bf g})$.  We denote by $z_i$ one-half of the square norm function defined by $h_i$ on $V_i$. Thus, $z_i$ is the  fibre-wise momentum map for the standard $U(1)$ action on $V_i$  by scalar multiplication, with respect to the  imaginary part of the hermitian product  defined by $h_i$. We consider the fibre-wise K\"ahler quotient at  moment value $z_0+ \cdots z_{\ell}=1$ of $$V= \bigoplus_{k=0}^{\ell} V_k$$  with respect to the hermitian product  $h= h_0 \oplus \cdots \oplus h_{\ell}$ and the diagonal $U(1)$ action on $V$: this gives the Fubini--Study metric $g^{\rm FS}_{\PP(V)}$ of scalar curvature $2n(n-1)$  
on each fibre of $M_0=\PP(V)\to C$. We use the  Chern connection of $(V, h)$ (which induces a horizontal distribution on $TM_0$) in order to complete trivially $(g^{\rm FS}_{\PP(V)}, \omega^{\rm FS}_{\PP(V)})$ in the horizontal direction,   and thus define a K\"ahler metric on $M_0$ as follows:
\begin{equation}\label{gc-1}
\begin{split}
g_c &= \Big(c - \sum_{k=0}^{\ell} \mu(V_k)L_k(x)\Big)g_C  + g^{\rm FS}_{\PP(V)}, \\
\omega_c &=  \Big(c - \sum_{k=0}^{\ell} \mu(V_k)L_k(x)\Big)\omega_C + \omega^{\rm FS}_{\PP(V)}, 
\end{split}
\end{equation}
where:
\begin{enumerate}
\item[$\bullet$]   the function $L_j(x)$ is the restrictions of $z_j$ on the level set $z_0 + \cdots z_{\ell}=1$ and then quotient to $M$; letting $x_i:= z_i = L_i(x)$ for $i=1, \ldots \ell$, we then have $L_0(x)= 1- \sum_{j=1}^{\ell}x_j$. Thus, $(x_1, \ldots, x_{\ell})$ is the induced (fibre-wise) moment map for the $\T^{\ell}$ action on $(\PP(V), \omega_c)$,  taking values in the standard simplex $\Delta \subset \R^{\ell}$.
\item[$\bullet$] $c$ is a real constant satisfying
\begin{equation}\label{c-constraint}
c>{\rm max}\{\mu(V_i), i=0, \ldots, \ell\},
\end{equation}
or,  equivalently, $(c - \sum_{k=0}^{\ell} \mu(V_k)L_k(x))>0$ on $M_0$.
\item[$\bullet$] $(g_C, \omega_C)$  is the pull back of the K\"ahler structure on $C$ to $M_0$.
\end{enumerate}
It is not immediately clear from the above description that $\omega_c$ is a closed form, but for various  computational purposes it will be more convenient to describe $(g_c, \omega_c)$ in terms of its pull-back to the the blow-up ${\hat M_0}$ of $M_0$  along the  sub-manifolds $S_i=\PP(V_i) \subset \PP(V)$, which  is isomorphic to the total space of the $\C \PP^{\ell}$ fibre-bundle 
$${\hat M}_0 =  \PP\Big(\cO(-1)_{\PP(V_0)} \oplus  \cdots \oplus \cO(-1)_{\PP(V_{\ell})}\Big) \to \hat S,$$ 
where 
$$\hat S = \PP(V_0)\times _C \cdots \times _C \PP(V_{\ell}) \to C.$$
We can summarize the setting in the following commutative diagram
\begin{equation*}
\begin{CD}
\hat M_0= \PP\Big(\cO(-1)_{\PP(V_0)}\oplus  \cdots \oplus \cO(-1)_{\PP(V_{\ell})}\Big) @>>> {\hat S}=\PP(V_0)\times_C \cdots \times_C \PP(V_{\ell})\\
 @VVV  @VVV\\
M_0= \PP\Big(V_0\oplus \cdots \oplus V_{\ell}\Big) @>>> C
\end{CD}
\end{equation*}
Notice that $\hat S$  admits  a  family of  (locally symmetric) CSC K\"ahler metrics of the form 
$$g_{{\bf a}, b}^{\hat S} = \sum_{k=0}^{\ell}a_k g^{\rm FS}_{V_k} + b g_C, $$
where ${\bf a}=(a_0, \ldots, a_{\ell})$ is an $(\ell+1)$-tuple of positive real numbers, $b>0$ and $g^{\rm FS}_{V_i}$ denotes the Fubini--Study metric of scalar curvature $2{\rm rk}(V_i)({\rm rk}(V_i)-1)$ defined on the fibres of $\PP(V_k)$ by using the hermitian product  $h_k$,  and on $\hat S$ by using the projectively flat structure of $(V_k,h_k)$.

We  denote by ${\hat \theta}_i$ the (real-valued) connection 1-form on the unitary bundle $P_i\subset \cO(-1)_{\PP(V_i)}$ with respect to $h_i$, induced via  the Chern connection of $(\cO(-1)_{\PP(V_i)}, h_i)$. Using that the curvature of $(V_i,h_i)$ is $\mu(V_i){\rm Id}\otimes \omega_C$,  $\hat \theta_i$ satisfies
\begin{equation}\label{theta-i}
\hat \theta_i (K_i)=1, \ \ d{\hat \theta} _i  =  \omega_{V_i}^{\rm FS} -  \mu(V_i) \omega_C,
\end{equation}
where $K_i$ stands for the generator of the standard $S^1$ action on $\cO(-1)_{\PP(V_i)}$,   and $\omega_{V_i}^{\rm FS}$ and $\omega_C$ are  the $(1,1)$-forms associated to the tensors $g^{\rm FS}_{V_i}$ and $g_C$ on $\hat S$,   introduced above. 

Using arguments identical to \cite[Lemma~1]{ACGT-0} (see also \cite[Theorem. 2]{ACGT-1} and \cite[Section 2.3]{ACGT}), one can see that the pull-back of $(g_c, \omega_c)$ to  ${\hat M}_0 \to \hat S$ is given by
\begin{equation}\label{generalized-Calabi-all}
\begin{cases}
g_c = &\sum_{k=0}^{\ell}L_k(x) g_{V_k}^{\rm FS} + \Big(c - \sum_{k=0}^{\ell} \mu(V_k)L_k(x)\Big)g_C \\
        & + \sum_{i,j=1}^{\ell} \Big(({\rm Hess}(u))_{ij} dx_idx_j  + ({\rm Hess}(u))^{-1}_{ij}{\theta}_i {\theta}_j\Big),\\
 \omega_c = &  \sum_{k=0}^{\ell}L_k(x) \omega_{V_k}^{\rm FS}  +\Big(c - \sum_{k=0}^{\ell} \mu(V_k)L_k(x)\Big)\omega_C \\
        & + \sum_{j=1}^{\ell} dx_j \wedge {\theta}_j,      
\end{cases}
\end{equation}
where:
\begin{enumerate}
\item[$\bullet$] $L_j(x) = x_j, \ j=1, \ldots \ell; \ L_0(x) = 1- \sum_{j=1}^{\ell}x_j$ and $x=(x_1, \ldots x_k)$ belongs to the standard simplex $\Delta=\{x : L_k(x) \ge 0, k=0, \ldots \ell \} \subset {\mathbb R}^{\ell}$;
\item[$\bullet$] $\theta_j= {\hat \theta}_j - {\hat \theta}_0, j=1, \ldots, \ell$ are the components of a connection $1$-form  defined on a principle $\T^{\ell}$ bundle $P$ over $\hat S$,  such that
\begin{equation}\label{theta}
d\theta_j = \omega_{V_j}^{\rm FS} - \omega_{V_0}^{\rm FS} + (\mu(V_0)-\mu(V_j))\omega_C, j=1, \ldots, \ell;
\end{equation}
\item[$\bullet$]  $u= \frac{1}{2}\sum_{k=0}^{\ell} L_k(x) \log L_k(x)$
is the Guillemin potential for the Fubini--Study metric on the $\C \PP^{\ell}$-fibre of $\hat M_0 = \PP\Big(\bigoplus_{k=0}^{\ell} \cO(-1)_{\PP(V_k)}\Big) \to \hat S$.
\end{enumerate}
The metric \eqref{generalized-Calabi-all} is a special case of the  generalized Calabi construction developed in \cite{ACGT-1, ACGT}.  For the purpose of computing of Donaldson--Futaki invariant, we shall work with the form \eqref{generalized-Calabi-all} of the metric, and this can be merely taken to be its definition: even though \eqref{generalized-Calabi-all}-\eqref{theta} define a degenerate K\"ahler metric on $\hat M_0$,  it is shown in \cite[Prop. 2 and Theorem 2]{ACGT-1}  that it is the pull-back of a smooth K\"ahler metric on $M_0= \PP(V) \to C$, provided that condition \eqref{c-constraint} is satisfied.  The corresponding K\"ahler class $\Omega_c=[\omega_c]$ on $M$ is  called  {\it admissible}.  The definition \eqref{gc-1} yields  that $(g_c, \omega_c)$ restricts to each fibre of $\PP(V)$ to a Fubini--Study metric of  scalar curvature $2n(n-1)$, thus showing that  $\Omega_c= 2\pi \left(c_1(\cO(1)_{\PP(V)}) + m c_1(\cO(1)_C)\right)$ for a certain  real number $m$.  This can also be deduced directly  from \eqref{generalized-Calabi-all}, for instance by integrating $\omega_c^{n-1}$ over a fibre of the fibration (over $C$)  $\hat M_0 \to \hat S \to C$ (and using Lemma~\ref{techn} below).  We claim that $m=c$. To show this  we use \cite[Lemma 5.16]{RT} to compute (denoting $r_V=\rk(V)=n$)
\begin{equation*}
\frac{1}{(2\pi)^{n}} \int_{M_0}\frac{{\Omega_c}^{r_V}}{r_V!}=\frac{1}{r_V!}\Big(c_1(\cO(1)_{\PP(V)}) + m c_1(\cO(1)_C)\Big)^n= \frac{1}{r_V!}\big(r_V m -d_V\big),
\end{equation*}
on the one hand,  and  Proposition~\ref{p:computation} below to get
\begin{equation*}\label{alpha0eqn}
\frac{1}{(2\pi)^{n}} \int_{M_0}\frac{{\Omega_c}^{r_V}}{r_V!}  = \frac{\alpha_0}{\pi_R}= \frac{1}{r_V!}\big(r_Vc - d_V\big)
 \end{equation*}
 on the other hand.

Conversely, as $H^2(M_0, \R) \cong \R^2$, any K\"ahler class $\Omega$ on $M_0$ can be rescaled by a positive real number so as it becomes of the form $[\omega_c]$ for some real number $c$ (possibly not satisfying \eqref{c-constraint}). However,  integrating suitable powers of $\omega_c$ over the sub-manifolds $S_i=\PP(V_i) \subset M_0$ ($S_i$ is the pre-image of a vertex of $\Delta$) yields the inequality \eqref{c-constraint}. This shows that {\it any} K\"ahler class on $M_0$ is admissible up to a scale. We have thus established
\begin{lemma}\label{l:kahler-class} Let $M_0=\PP(V)\to C$ with $V=\bigoplus_{i=0}^{\ell} V_i,$  where $V_i$ are stable vector bundle over a curve $C$. Then,  \eqref{generalized-Calabi-all}-\eqref{theta}-\eqref{c-constraint} introduces  K\"ahler metrics  $(g_c, \omega_c)$ on $M_0,$ which exhaust the K\"ahler cone of $M_0$ up to positive scales. The constant $c$ corresponding to a  polarization $\cL_{m}^0=\cO(1)_{\PP(V)} \otimes \cO(m)_C$ on $M_0$ is $c=m$.
\end{lemma}

\subsection{Computing the relative Donaldson--Futaki invariant via $g_c$}\label{computgc} 
We shall start this section by fixing some notation.
\begin{notation}\label{notation1}
We denote for all $i=0,..,\ell$
\begin{align*}
r_i=\rk(V_i),& & d_i=\deg(V_i),& & \mu_i=\mu(V_i),
\end{align*}
and 
\begin{align*}
r_V=r_0+...+r_\ell, & & d_V=\deg(V)=\sum_{i=0}^{\ell} d_i, & & \pi_R=&(r_0-1)!...(r_\ell-1)!.
\end{align*}
\end{notation}

The volume form $v_{g_c}= \omega_c^{n}/n!$ (with $n=r_V$ being the complex dimension of $M$) of the metric \eqref{generalized-Calabi-all} is given by
$$v_g = \frac{p_c(x)}{\pi_R} \Big[\omega_C\wedge \Big(\bigwedge_{i=0}^{\ell} (\omega_{V_i}^{\rm FS})^{(r_i-1)}\Big)\wedge \Big(d\mu \wedge \theta_1\wedge \cdots \wedge \theta_{\ell}\Big)\Big], $$
where $d\mu$ is the standard Lebesgue measure on $\R^{\ell}$ and we have set 
\begin{equation}\label{pc}
 p_c(x) = \Big(c-\sum_{k=0}^{\ell} \mu_k L_k(x)\Big)\prod_{k=0}^{\ell}(L_k(x))^{(r_k-1)},
 \end{equation}
The scalar curvature ${\rm Scal}_{g_c}$ of the metric \eqref{generalized-Calabi-all} is computed in \cite{ACGT} to be
\begin{equation}\label{Scal}
\begin{split}
{\rm Scal}_{g_c} &= \sum_{k=0}^{\ell} \frac{2r_k(r_k-1)}{L_k(x)} + \frac{4(1-{\bf g})}{\big(c - \sum_{k-0}^{\ell} \mu_k L_k(x)\big)} \\
                              & - \frac{1}{p_c(x)}\sum_{p,q=1}^{\ell} \frac{\partial^2}{\partial x_p \partial x_q}\big(p_c(z)u^{pq}(x)\big),
                   \end{split}
\end{equation}
 where $(u^{pq}(x))$ denotes $({\rm Hess}(u))^{-1}$.   We then compute (by using integration by parts, compare with  \cite[Section 2.5]{ACGT}): 
 \begin{equation}\label{alphas}
 \begin{cases}
 \alpha_0&:=  \frac{\pi_R}{(2\pi)^{n}}\int_{M_0} v_{g_c}= \int_{\Delta} p_c(x) d\mu,  \\
 \alpha_r&: = \frac{\pi_R}{(2\pi)^{n}}\int_{M_0} x_r v_{g_c}=\int_{\Delta} x_r p_c(x) d\mu, \\
\alpha_{rs}&: = \frac{\pi_R}{(2\pi)^{n}}\int_{M_0} x_r x_s v_{g_c} =\int_{\Delta} x_rx_s p_c(x)d\mu, \\
\beta_0  &: = \frac{\pi_R}{2(2\pi)^{n}}\int_{M_0}{\rm Scal}_{g_c} v_{g_c}\\
                & = \int_{\Delta}\Big(\frac{2(1-{\bf g})}{c-\sum_{k=0}^{\ell}\mu(V_k)L_k(x)}  + \sum_{k=0,r_k\neq 1}^{\ell} \frac{{\rm rk}(V_k)({\rm rk}(V_k)-1)}{L_k(x)}\Big)p_c(x) d\mu \\
                 &\hfill + \int_{\partial \Delta} p_c(x) d\sigma, \\
\beta_r  &:=  \frac{\pi_R}{2(2n)^{n}}\int_{M_0}{\rm Scal}_{g_c} x_r v_{g_c}\\
                 &=\int_{\Delta}\Big(\frac{2(1-{\bf g})}{c-\sum_{k=0}^{\ell}\mu(V_k)L_k(x)}  + \sum_{k=0,r_k\neq 1}^{\ell} \frac{{\rm rk}(V_k)({\rm rk}(V_k)-1)}{L_k(x)}\Big)x_r p_c(x) d\mu \\
                 & \hfill+ \int_{\partial \Delta} x_r p_c(x) d\sigma,  
 \end{cases}
\end{equation}
where $d\sigma$ is the induced measure on the facets of $\Delta$ by $u_j \wedge d\sigma_i = - d\mu$ for each facet $F_j$ with $u_j=dL_j$ being the inward normal of $F_j$.

\smallskip
We obtain  that the normalized hamiltonians  and  ${\rm Scal}_{g_c}^{\T^{\ell-1}}= f^M_{\rm ex}$ are given respectively by 
$$f_k= x_k - \frac{\alpha_k}{\alpha_0}, \ \ f^M_{\rm ex}= a_0 + \sum_{j=2}^{\ell}a_j x_j$$ 
with 
\begin{equation}\label{syst}
\begin{cases}
 a_0 \alpha_k + \sum_{j=2}^{\ell} a_j \alpha_{jk} &= 2\beta_k, \ \  k=2, \ldots \ell, \\  
 a_0 \alpha_0 + \sum_{j=2}^{\ell} a_j \alpha_j  &= 2\beta_0.
 \end{cases}
 \end{equation}
This allows us to obtain from \eqref{relative-futaki-3}
\begin{equation*}
\begin{split}
\frac{\alpha_0\pi_R}{(2\pi)^n} \mathfrak{F}_{\rho_{\rm ex}}(\rho_{L}) &= 2\left(\alpha_0\beta_{1} - \alpha_{1}\beta_0 \right) - \sum_{j=2}^{\ell} a_j \alpha_0\Big(\alpha_{j1} - \frac{\alpha_{1}\alpha_j}{\alpha_0}\Big) \\
                                                                  &= 2\left(\alpha_0\beta_{1} - \alpha_{1}\beta_0 \right)  -2\sum_{j,r=2}^{\ell} (A^{-1})_{rj} \Big(\alpha_0\beta_r -\alpha_r\beta_0\Big)\Big(\alpha_{j1} - \frac{\alpha_{1} \alpha_j}{\alpha_0}\Big),
\end{split}
\end{equation*}
where we introduced the matrix of size $\ell-1$,
\begin{equation}\label{matrixA} A_{ij}= \alpha_{ij} -\frac{\alpha_i\alpha_j}{\alpha_0},  \ \ 2\le i,j\le \ell.
\end{equation}
We thus have, setting $\mathring{\mathfrak{F}}_{\rho_{\rm ex}}(\rho_{L})=\frac{1}{2}\frac{\alpha_0\pi_R}{(2\pi)^n} \mathfrak{F}_{\rho_{\rm ex}}(\rho_{L})$,
\begin{lemma}\label{l:futaki-computed} Let $M=\PP(E)\to C$ with $E=\bigoplus_{k=0}^{\ell-1}U_k$ be a projectivisation of vector bundle over a curve $C$ of genus ${\bf g}\ge 2$ and $L\subset  U_0$ a strict sub-bundle of one of the indecomposable components of $E$. The relative Donaldson--Futaki invariant $\mathfrak{F}_{\rho_{\rm ex}}(\rho_L)$ of the induced $\C^{\times}$ action $\rho_L$ on the central fibre $M_0=P(V)$ with respect to the test configuration of Lemma~\ref{l:test-configuration} and a polarization $\cL_{m}$  is positive multiple  of 
\begin{equation}\label{futaki-computed}
\mathring{\mathfrak{F}}_{\rho_{\rm ex}}(\rho_{L}) := (\alpha_0\beta_1-\alpha_1\beta_0)-\sum_{j,r=2}^{\ell} (A^{-1})_{rj} \Big(\alpha_0\beta_r -\alpha_r \beta_0\Big)\Big(\alpha_{j1} - \frac{\alpha_{1} \alpha_j}{\alpha_0}\Big),
\end{equation}
where $\alpha_i, \alpha_{ij}$ and $\beta_i$ are the integrals defined by \eqref{alphas} with $c=m$,  and $A$ is the matrix  \eqref{matrixA}.
\end{lemma}

\smallskip In the remainder of this section, we collect the main technical ingredients allowing to evaluate the sign of the r.h.s. of \eqref{futaki-computed}

\begin{notation}
We denote
\begin{itemize}
 \item $\kappa=\#\{j,r_j=1\}$ the number of integers $0\leq j\leq \ell$ such that $r_j=1$,
\item $\kappa_k=\#\{j\neq k,r_j=1\}$ the number of integers $0\leq j\neq k\leq \ell$ such that $r_j=1$,
\item $\kappa_{k_1,k_2}=\#\{j\neq k_1,j\neq k_2,r_j=1\}$ the number of integers $0\leq j\neq (k_1,k_2)\leq \ell$ such that $r_j=1$.
\end{itemize}
\end{notation}

\begin{prop}\label{p:computation} With the notations above, $j\neq k$, $0<j,k<l$, we have 
\begin{align*}
\alpha_0=&\frac{\pi_R}{r_V!}(cr_V - d_V),\\
\alpha_j=&\frac{\pi_R}{(r_V+1)!}r_j\left(c(r_V+1) -d_V-\mu_j\right),\\
\alpha_{jk}=&\frac{\pi_R}{(r_V+2)!}r_jr_k(c(r_V+2)-d_V -\mu_j -\mu_k), \\
\alpha_{jj}=&\frac{\pi_R}{(r_V+2)!}r_j(r_j+1)(c(r_V+2)-d_V -2\mu_j),\\
\beta_0=&\frac{\pi_R}{(r_V-1)!}\left( (r_V-1)r_V c +2(1-{\bf g})  - (r_V-1)d_V\right),\\
\beta_j=&\frac{\pi_Rr_{j}}{r_V!} \left( r_V(r_V-1)c +2(1-{\bf g})- d_V(r_V-2)-r_V\mu_j \right).
\end{align*}
\end{prop}
\begin{proof}
This is a direct corollary of Lemmas \ref{techn} and \ref{kappa} that can be found in the Appendix (Section \ref{appendix}).  
Actually, we have
$$\alpha_0=\frac{\pi_R}{(r_V-1)!}\left(c - \frac{\mu_0r_0+...+\mu_\ell r_\ell}{r_V}\right).$$
For $j>0$, we get
$$\alpha_j=\frac{\pi_R}{(r_V-1)!}\frac{r_j}{r_V}\left(c - \frac{\mu_0r_0+..+\mu_j(r_j+1)+..+\mu_\ell r_\ell}{r_V+1}\right).$$
If $j\neq k$, then
\begin{align*}\alpha_{jk}=&\alpha_{kj}\\
             =&\frac{\pi_R}{(r_V-1)!}\frac{r_jr_k}{(r_V)(r_V+1)}\left(c - \frac{\mu_0r_0+..+\mu_j(r_j+1)+..+\mu_k(r_k+1)+..+\mu_\ell r_\ell}{r_V+2}\right) 
\end{align*}
If $j= k$,
$$\alpha_{jj}=\frac{\pi_R}{(r_V-1)!}\frac{r_j(r_j+1)}{(r_V)(r_V+1)}\left(c - \frac{\mu_0r_0+..+\mu_j(r_j+2)+..+\mu_\ell r_\ell}{r_V+2}\right).$$
Moreover,
\begin{align*}\beta_0=&\frac{\pi_R}{(r_V-1)!} \left(2(1-{\bf g}) +  (r_V-1) \sum_{k=0, r_k\neq 1}^{\ell} r_k \left(c - \frac{\mu_0r_0+..+\mu_k(r_k-1)+..+\mu_\ell r_\ell}{r_V-1}\right)\right)\\
&+\frac{\pi_R}{(r_V-1)!}\left((r_V-1)\kappa c-\sum_{k=0}^\ell d_k\kappa_k\right)\\
  =&\frac{\pi_R}{(r_V-1)!} \left(2(1-{\bf g}) +  (r_V-1) \sum_{k=0}^{\ell} r_k \left(c - \frac{\mu_0r_0+..+\mu_k(r_k-1)+..+\mu_\ell r_\ell}{r_V-1}\right)\right)\\
  &-\frac{\pi_R}{(r_V-1)!}\left((r_V-1)c \kappa - \sum_{k=0,r_k=1}^\ell (d_0+...+d_{k-1}+d_{k+1}...+d_{\ell})\right)\\
  &+\frac{\pi_R}{(r_V-1)!}\left((r_V-1)\kappa c-\sum_{k=0}^\ell d_k\kappa_k\right)\\ 
=&\frac{\pi_R}{(r_V-1)!}\left(2(1-{\bf g}) +  (r_V-1)r_V c - r_Vd_V+ \sum_{k=0}^{\ell} r_k \mu_k \right).
\end{align*}
Similarly, for $j>0$,
{
\begin{align*}\beta_j
=&\frac{\pi_Rr_{j}}{r_V!} \left(2(1-{\bf g}) + r_V(\sum_{k\neq j}^{\ell} r_k)c - \sum_{k\neq j}^{\ell} r_k(d_V-\mu_k+\mu_j)\right)+
\frac{\pi_R}{r_V!}(cr_V - d_V)r_j(r_j-1)\\
&-\frac{\pi_Rr_{j}}{r_V!}\left(r_V\kappa_jc -\sum_{k=0,k\neq j}^\ell (d_0+...+d_{k-1}+d_{k+1}+...+(d_j+\mu_j)+...+d_\ell)\right)\\
&+\frac{\pi_R}{r_V!}\left(r_jr_V\kappa_j c- r_j\sum_{k=0}^\ell d_k\kappa_{k,j}-\kappa_jd_j\right)\\
=&\frac{\pi_Rr_{j}}{r_V!} (2(1-{\bf g}) + r_V(r_V-r_j)c - (d_V+\mu_j)(r_V-r_j)+d_V-d_j +(r_j-1)(cr_V-d_V)).
\end{align*}
}
\end{proof}

We need to compute the term $\alpha_0\alpha_{jk}-\alpha_j\alpha_k$ explicitly in order to get $\mathfrak{F}_{\rho_{\rm ex}}(\rho_{L})$. By direct computation from the previous proposition, we obtain

\begin{lemma} \label{gamma}
Define
\begin{alignat*}{3}
 \gamma_{jk}&=&\frac{\pi_R^2r_{j}r_k}{(r_V+1)!^2(r_V+2)}[-(r_V+1)(r_V+2)c^2+2(\mu_k+\mu_j+d_V)(1+r_V)c\\
&&-(\mu_k+\mu_j+d_V)d_V-(r_V+2)\mu_j\mu_k],&\\
\gamma'_{j}&=& \frac{\pi_R^2}{r_V! (r_V+2)!}r_j[r_V(r_V+2)c^2-2(d_V(r_V+1)+r_V\mu_j)c+d_V(2\mu_j+d_V)].&
\end{alignat*}
Then for $j\neq k$, 
 \begin{align*}
\alpha_0\alpha_{jk}-\alpha_j\alpha_k=\gamma_{jk}, \hspace{1cm} \alpha_0\alpha_{jj}-\alpha_j\alpha_j=\gamma_{jj}+\gamma'_j.
\end{align*}
\end{lemma}

In a similar way, we obtain the following lemma.

\begin{lemma}\label{betakalpha0}
We have
 \begin{align*}\beta_k\alpha_0-\beta_0\alpha_k &= \frac{\pi_R^2}{r_V!(r_V+1)!} \left(2r_V(r_kd_V-r_Vd_k)c +2({\bf g}-1-d_V)(r_kd_V-r_Vd_k)\right),\\
\end{align*}
and in particular
\begin{align*}
 \sum_{k=2}^{\ell}\beta_k\alpha_0-\beta_0\alpha_k  =&2 c   \frac{\pi_R^2}{(r_V-1)!(r_V+1)!} \left( 
 r_V(d_0+d_1)-d_V(r_0+r_1) \right)  \\ 
 &+2\frac{\pi_R^2}{r_V!(r_V+1)!} ({\bf g}-1-d_V)\left( 
 r_V(d_0+d_1)-d_V(r_0+r_1) \right).
 \end{align*}
\end{lemma}

\subsection{Algebraic computation of the relative Donaldson--Futaki invariant}\label{sectDF}  We consider in this section  $M=\PP(\bigoplus_{i=0}^{\ell}V_i)\to C$ with no assumptions for $V_i$ or $C$, and take a polarization $\cL=\cL_{q,p}= \cO(q)_{\PP(E)}\otimes \cO(p)_C$. Up to scale, these will only  depend on the ratio $p/q$, so write $\cL=\cL_m=\cO(1)_{\PP(V)}\otimes \cO(m)_C$ with $m\in \Q$. Using that the volume of the sub-variety $\PP(V_i) \subset \PP(V)$ with respect to $\cL$ must be  positive for $\cL$ to define a polarization, one gets (see e.g. \cite[Prop.~1]{fujiki}) that $m>\mu(V_i)$ for all $i=0, \ldots, \ell$, compare with \eqref{c-constraint}.

We denote by $\rho_i, i=1, \ldots, r$ the $\C^{\times}$ action on $M$ given by multiplication on $V_i$ (and acting trivially on the other summands of $V=\bigoplus_i V_i$).  We want to compute algebraically the relative Donaldson--Futaki invariants of $\rho_i$ on $(M, \cL)$.  This computation for the classical Donaldson--Futaki invariant is a standard procedure and can be done in different ways, see \cite[Section 5.4]{RT} and \cite{DV1, Fut, KR1,Sz1}. 

\smallskip
We first compute $d_k=\chi(\PP(V),\mathcal{L}^k)$ for $k\gg 0$ using Proposition \ref{technique}. Recall that $n=\dim(\PP(V))=r_V$ is the dimension of the ruled manifold. With our notations, we have the formula $\pi_*\mathcal{L}^k =S^{k}V^* \otimes \cO(mk)_C$. In the computations below, we will also use the fact that $\int_C c_1(C)=2(1-{\bf g})$ and ${\rm deg}_C \cO(1)_C=1$. Then, we have 
 \begin{align*}
 d_k=&\int_C ch(S^{k}V^*) ch(\mathcal{O}_C(mk) ) Todd(C),\\
 =&\binom{n-1+k}{k}\Big(k(m -\mu(V))+1-{\bf g}\Big),\\
 =& \tilde \alpha_0k^n + \tilde \beta_0 k^{n-1} + O(k^{n-2}),
\end{align*}
with
\begin{align*}
\tilde  \alpha_0=&\frac{1}{(n-1)!}\left(m -\mu(V)\right)=\frac{1}{r_V!}\left(r_Vm- d_V\right) =\frac{1}{\pi_R}\alpha_0(m),\\
\tilde \beta_0 =&\frac{1}{(n-1)!}\Big((1-{\bf g})+\frac{n(n-1)}{2}(m-\mu(V))\Big)\\
                        =& \frac{1}{2(r_V-1)!}\big((r_V-1)(r_V m -d_V) + 2(1 -{\bf g})\Big) =  \frac{1}{2\pi_R}\beta_0(m).
 \end{align*}
 where $\alpha_i(m), \beta_i(m)$ are given by Proposition~\ref{p:computation} with $c=m$.  
 
 \smallskip
For  a $\C^{\times}$ action  $\rho$ on $(M, \cL)$, there is an associated weight $w_k(\rho)$ given by the trace  $\tr(A_k)$ of the infinitesimal generator $A_k$ on $H^0(M,\mathcal{L}^k)$, see Definition~\ref{d:convention}. For $k\gg 0$,  we let
$$w_k(\rho_i) =  k^{n+1} \tilde \alpha_i  + k^{n}\tilde \beta_i + O(k^{n-1}).$$
In order to compute  $w_k(\rho_i)$, we apply the $S^1$-equivariant Riemann--Roch theorem with the  Cartan model of equivariant cohomology in order to compute the equivariant characteristic quantities.

For so, let $\T^{r+1}$ denotes the natural $(r+1)$-dimensional torus action by scalar multiplication on each factor $V_i$ and $\rho$ be the $\C^{\times}$ action associated to an $S^1$ subgroup of $\T^{r+1}$:
$$\rho(t)\cdot (v_0,\cdots,v_r)=(t^{\lambda_0}v_0,\cdots, t^{\lambda_{r}}v_{r})$$
for some integer coefficients $\lambda_i$. Let us fix  a $\T^{r+1}$-invariant hermitian metric $h= h_0\oplus \ldots \oplus h_r$  on $V=\bigoplus_{i=0}^r V_i$  (here $h_i$ is a fixed hermitian metric on $V_i$) with $\rho$ equivariant curvature $\frac{1}{2\pi}F_V - \Lambda$, where  $F_V=F_0\oplus \cdots \oplus F_r$ is the usual (non-equivariant) curvature of $V$ (with $F_i$ being the curvature of  $(V_i,h_i)$),  and $\Lambda=\Lambda_{\rho}$ is the endomorphism of $E$ given by $$\Lambda_{\rho}=\Lambda=\lambda_0 Id_0\oplus \ldots \oplus \lambda_r Id_r.$$  Notice that the sing $-$ in front of $\Lambda$ of the equivariant curvature corresponds to our convention in Definition~\ref{d:convention} for the infinitesimal generator of the action $\rho$ on $V$.  Thus,  $-\frac{1}{2\pi}F_V +\Lambda$ is a $\rho$-equivariant curvature for the dual action on $V^*$ (still denoted by $\rho$). Using the  identification $\pi_*\mathcal{L}^k =S^{k}V^* \otimes \cO(mk)_C$, we apply the $S^1$-equivariant Riemann--Roch theorem (see \cite{AB} and \cite{BV}) in order to compute $w_k(\rho)$, as is done in \cite{Do2}.  Since we are dealing with $S^1$-invariance over a base of dimension 1, it is only necessary to compute the $(2,2)$ part of the $\rho$-equivariant cohomology class  
 \begin{equation}\label{eq1}ch^{\rho}(S^{k}V^*)ch^{\rho}(\mathcal{O}_C(mk))Todd^{\rho}(C)\end{equation}
in order to get the weight $w_k(\rho)$ by integration.  We can apply Proposition~\ref{technique} together with the fact that 
$$Todd^{\rho}(C)=1+ \frac{1}{2}c_1(C) =[1 +(1-{\bf g}) \frac{\omega_C}{2\pi}],$$ 
where $\omega_C$ is any representative of $c_1(\cO(1)_C)$, in order to expand \eqref{eq1}. Then, using equivariant Chern--Weil theory, we can replace before integration the quantities $c^{\rho}_1(V^*)$, $c^{\rho}_2(V^*)$, $c^{\rho}_1(\cO(1)_C),$ using the following formulas:
 \begin{align*}
  c^{\rho}_1(V^*)=&  [-\frac{1}{2\pi}\tr(F_V)+\tr(\Lambda)], \\
  ch_2^{\rho}(V^*)=& [-\frac{1}{2\pi}\tr(F_V\Lambda)+ \frac{1}{2}\tr(\Lambda^2)],\\
  c^{\rho}_1(\cO(1)_C)=&[\frac{1}{2\pi}\omega_C].
 \end{align*}
We obtain, keeping only the terms that can be integrated along $C$, 
 \begin{align*}
  w_k(\rho)=&-\frac{1}{2\pi} \int_C \Big[\binom{n+k}{k-1} \tr(F_V\Lambda) + \binom{n-1+k}{k-2}\tr(\Lambda)\tr(F_V) \\
                     &\pushright{- \binom{n-1+k}{k-1}\tr(\Lambda)(1-{\bf g})\omega_C - \binom{n-1+k}{k-1} k m \tr(\Lambda)\omega_C)\Big]}\\
                   = &  - \binom{n+k}{k-1} \Big(\sum_{i=0}^r \lambda_i d_i\Big) -  \binom{n-1+k}{k-2}\big(\sum_{i=0}^r \lambda_i r_i\big)d_V \\
                    & \pushright{+ \binom{n-1+k}{k-1}\big(\sum_{i=0}^r \lambda_i r_i\big)\big((1-{\bf g}) + km\big)},
 \end{align*}
where $r_j={\rm rk}(V_j)$; $d_j = {\rm deg}_C (V_j)$.
Letting $\rho=\rho_i$ (i.e. $\lambda_j=\delta_{ij}$), we thus  get 
$ w_k(\rho_i)= \tilde \alpha_i k^{n+1}+\tilde \beta_i k^{n}+ O(k^{n-1})$ with
\begin{align*}
\tilde \alpha_i =&\frac{1}{(n+1)!} (-d_i -r_i d_V +(n+1)mr_i),\\
                          =& \frac{r_i}{(r_V + 1)!}\Big(m(r_V+1) -d_V - \mu_i\Big),\\
                          =& \frac{1}{\pi_R}\alpha_i(m);  \\               
\tilde \beta_i =&-\frac{1}{(n+1)!}\Big(d_j\frac{n(n+1)} {2} + r_jd_V\Big) +\frac{1}{n!}\Big(r_i(1-{\bf g}) + m \frac{n(n-1)}{2}\Big),\\
                       =& \frac{r_i}{2 n!}\Big( 2(1-{\bf g}) + r_V(r_V-1) -r_V\mu_i -(r_V-2)d_V\Big),\\
                       =&  \frac{1}{2\pi_R}\beta_i(m),
\end{align*}
with $\alpha_j(m), \beta_j(m)$ given by Proposition~\ref{p:computation} with $c=m$.

\smallskip
Similarly, letting $w_k(\rho_i, \rho_j)$ denote the trace $\tr(A_{k,i} A_{k,j})$ where  $A_{k,i}$ is the infinitesimal generator  of the actions of $\rho_i$  on $H^0(M,\mathcal{L}^k) \cong  H^0(C, S^{k}V^* \otimes \cO(mk)_C)$, there is an expansion 
$$ w_k(\rho_i, \rho_j) = k^{n+2} \tilde \alpha_{ij} +  O(k^{n+1})$$
which we are going to detail below. We do a similar computation as before but apply the Hirzebruch--Riemann--Roch $S^1\times S^1$-equivariant Theorem to take into account  the two actions $\rho,\rho'$ corresponding to the generators of the $S^1\times S^1$ action. This time, working on $C\times S^1\times S^1$, we need to compute the $(3,3)$ part of 
 \begin{equation}\label{eq1a}ch^{\rho,\rho'}(S^{k}V^*)ch^{\rho,\rho'}(\cO)(mk)_C)Todd^{\rho,\rho'}(C)\end{equation}
 and integrate. This involves to compute the terms $T_1,T_2$ where
 \begin{align*}
  T_1=&\int_{C\times S^1\times S^1}ch_3^{\rho,\rho'}(S^{k}V^*),\\
  T_2=&\int_{C\times S^1\times S^1} ch_2^{\rho,\rho'}(S^{k}V^*)\left(c_1(\cO(mk)_C )+\frac{c_1(C)}{2}\right).
 \end{align*}
Since the base manifold is a curve, we use now that
 \begin{align*}
ch_3^{\rho,\rho'}(V^*) &= -\frac{1}{2\pi}[\tr(\Lambda_\rho \Lambda_{\rho'} F_V)], \\
c_1^{\rho,\rho'}(V^*)ch_2^{\rho,\rho'}(V^*) &=-\frac{1}{2\pi}[\tr(F_V \Lambda_\rho)\tr(\Lambda_{\rho'})+\tr(F_V \Lambda_{\rho'})\tr(\Lambda_\rho)+\tr(\Lambda_\rho\Lambda_{\rho'})\tr(F_V)], \\
c_1^{\rho,\rho'}(V^*)^3 &= -\frac{1}{2\pi}[6\tr(\Lambda_\rho)\tr(\Lambda_{\rho'})\tr(F_V)].  
 \end{align*}
We get,{
\begin{align*}
 T_1=\int_C&-\binom{n-1+k}{k-3}\tr(\Lambda_\rho)\tr(\Lambda_{{\rho'}})\tr(F_V)\\  
 &-(\tr(F_V\Lambda_{\rho})\tr(\Lambda_{{\rho'}})+\tr(F_V\Lambda_{{\rho'}})\tr(\Lambda_{\rho})+\tr(\Lambda_{\rho}\Lambda_{{\rho'}})\tr(F_V))\binom{n+k}{ k-2}\\
 &-\left(\binom{n+1+k}{k-1}+\binom{n+k}{k-2}\right)\tr(\Lambda_{\rho}\Lambda_{{\rho'}}F_V),\\
 T_2= \int_C&\binom{n+k}{k-1}\left(  \tr(\Lambda_\rho \Lambda_{\rho'})(mk\omega_C+\frac{1}{2}c_1(C))\right)\\
 &+\binom{n-1+k}{k-2}\tr(\Lambda_\rho)\tr(\Lambda_{\rho'})\left(mk\omega_C+\frac{1}{2}c_1(C)\right).\\
 \end{align*}
}
Thus,  the leading term of $\tr(A_kB_k)$  (which equals  the leading terms of $T_1+T_2$) is 
\begin{align*}
 \tr(A_kB_k)=& \frac{k^{n+2}}{2\pi(n+2)!}\int_C t_1+O(k^{n+1}),\\
 t_1=& -\tr(\Lambda_\rho)\tr(\Lambda_{\rho'})\tr(F_E)-\tr(F_E \Lambda_\rho)\tr(\Lambda_{\rho'})-\tr(F_E\Lambda_{\rho'})\tr(\Lambda_\rho)\\
 &-\tr(\Lambda_\rho\Lambda_{\rho'})\tr(F_E)-2\tr(\Lambda_\rho\Lambda_{\rho'} F_E) \\
 &+m{(n+2)}(\tr(\Lambda_\rho\Lambda_{\rho'})+\tr(\Lambda_\rho)\tr(\Lambda_{\rho'}))\omega_C.
\end{align*}
Letting $\rho=\rho_i$ and $\rho'=\rho_j$ as for the computation of $w_k(\rho_i)$, we obtain
\begin{align*}
\frac{1}{2\pi}\int_C t_1=& -(\sum_{s} \lambda_s r_s) ( \sum_t \lambda_t' r_t) d_V - ( \sum_{s} \lambda_sd_s ) ( \sum_t \lambda_t' r_t) -  ( \sum_{t} \lambda_t'd_t ) ( \sum_s \lambda_s r_s)\\
   &-(\sum_s \lambda_s\lambda_s' r_s)d_V - 2(\sum_s \lambda_s\lambda_s' d_s)\\
   &+m(n+2)[(\sum_s \lambda_s\lambda_s'r_s) + (\sum_{s} \lambda_s r_s)( \sum_t \lambda_t' r_t)]\omega_C.
\end{align*}
With $\lambda_s=\delta_{si}$ and $\lambda_t=\delta_{sj}$ we deduce from above,
\begin{align*}
\frac{1}{2\pi}\int_C t_1=& -r_i r_j d_V - d_ir_j -  d_j r_i +m(n+2)r_ir_j  &\text{ if }i\neq j,\\
\frac{1}{2\pi}\int_C t_1=& -r_j^2 d_V - 2 d_jr_j -r_jd_V-2d_j+m(n+2)(d_j +r_j^2)& \text{ if }i=j.
\end{align*}
Consequently, if $i\neq j$,
\begin{align*}
\tilde{\alpha}_{ij}=&\frac{1}{2\pi(n+2)!}\int_C t_1=\frac{1}{(n+2)!}r_ir_j[m(n+2)- d_V - \mu_i - \mu_j],\\
=&\frac{1}{\pi_R} \alpha_{ij}(m);\\
\tilde{\alpha}_{jj}=& \frac{1}{2\pi(n+2)!}\int_C t_1=\frac{1}{(n+2)!}r_j(r_j+1) [-d_V - 2\mu_j+m(n+2)],\\
=& \frac{1}{\pi_R}\alpha_{jj}(m),
\end{align*}
where, again, $\alpha_{ij}(m)$  are given by Proposition~\ref{p:computation} with $c=m$.

\smallskip
Recall from the general theory  (see Section~\ref{Preliminaries}) that the algebraic Donaldson--Futaki invariant $\mathfrak{F}(\rho_i)$ of $\rho_i$ on $(M, \cL_m)$ is given (up to a normalizing positive factor) by
$$\mathfrak{F}^{alg}(\rho_i)= \tilde \beta_i - \frac{\tilde \alpha_i \tilde \beta_0}{\tilde \alpha_0}.$$
If we assume now that $V_i$ are {\it indecomposable} and $C$ has genus ${\bf g}\ge 1$, so as the fibre-wise $\T^{\ell}$ action generated by $\rho_i, i=1, \ldots \ell$ corresponds to a maximal torus in ${\rm Aut}^{\rm red}(M)$, the {\it extremal} $\C^{\times}$ action $\rho_{\rm ex}$  of $(M, \cL_m)$ is  generated by 
$$K_{\rm ex}:= \sum_{i=1}^{\ell} \tilde{a}_i K_i, $$
where $K_i$ is a generator of $\rho_i$ and   the rational numbers $\tilde{a}_i$ are given by
$$\langle\rho_i,\rho_{\rm ex}\rangle=\sum_{k=1}^{\ell} \tilde{a}_k \Big(\tilde \alpha_{ik} -\frac{\tilde \alpha_i \tilde \alpha_k }{\tilde \alpha_0}\Big)=\mathfrak{F}^{alg}(\rho_i)= \tilde \beta_i - \frac{\tilde \alpha_i \tilde \beta_0}{\tilde \alpha_0},$$
see Section \ref{sectRK}.

\medskip

We now consider $M=\PP(\bigoplus_{i=0}^{\ell-1}U_i)\to C$ with $U_i$ indecomposable and $C$ of genus ${\bf g}\ge 1$, endowed with  a polarization $\cL_m=\cO(1)_{\PP(E)}\otimes \cO(m)_C$ and  the test configuration $\mathcal M$ with central fibre  $M_0=\PP (\bigoplus_{i=0}^{\ell} V_i)$ given by Lemma~\ref{l:test-configuration}. We use the above computations in order to express  the relative Donaldson--Futaki invariant $\mathfrak{F}^{alg}_{\rho_{\rm ex}^M}(\rho_L)$ on the central fibre $M_0$. For consistency, we let   $V_{i+1}:=U_{i}, i=1, \ldots, (\ell-1)$ and write the generic fibre as $M= \PP(U_0 \oplus \bigoplus_{i=2}^{\ell}V_i)\to C$. We  are also denoting by $\rho_i, i=2, \ldots, \ell$ the  $\C^{\times}$ actions on $M$ by multiplication on $V_i$ and let  $K_i, i=2, \ldots, \ell$ be the corresponding generating vector fields. By the discussion above, $\rho_{\rm ex}^M$ is the $\C^{\times}$ action generated by the vector field
$K^M_{\rm ex} = \sum_{i=2}^{\ell} a_i K_i$
with  $$\sum_{k=2}^{\ell} \tilde{a}_k \Big(\tilde \alpha_{ik} -\frac{\tilde \alpha_i \tilde \alpha_k }{\tilde \alpha_0}\Big) =\mathfrak{F}^{alg}(\rho_i)= \tilde \beta_i - \frac{\tilde \alpha_i \tilde \beta_0}{\tilde \alpha_0}.$$
This implies in particular that $\tilde{a}_k=a_k/4$ where $a_k$ satisfies \eqref{syst}.

\medskip
Now, the $\C^{\times}$ actions $\rho_i$ extend to the central fibre $M_0 = \PP(\bigoplus_{i=0}^{\ell} V_i)$, and the algebraic relative Donaldson--Futaki invariant on the central fibre $M_0$ is (see Definition~\ref{d:relative-futaki-gabor})
\begin{equation*}
\begin{split}
\mathfrak{F}^{alg}_{\rho_{\rm ex}}(\rho_L) &= \mathfrak{F}^{alg}_{\rho_{\rm ex}}(\rho_1)\\
&= \mathfrak{F}^{alg}(\rho_1)  - \langle\rho_1,\rho^M_{\rm ex}\rangle,\\
&=\mathfrak{F}^{alg}(\rho_1) - \sum_{j=2}^{\ell} \tilde{a}_j \Big(\tilde \alpha_{j1} -\frac{\tilde \alpha_j \tilde \alpha_1}{\tilde \alpha_0}\Big),\\
 &= \frac{1}{2\pi_R}\left(\beta_{1} - \frac{\alpha_{1}\beta_0}{\alpha_0} \right) - \frac{1}{4\pi_R}\sum_{j=2}^{\ell} a_j \Big(\alpha_{j1} - \frac{\alpha_{1}\alpha_j}{\alpha_0}\Big),\\
 &=  \frac{1}{4(2\pi)^n}\mathfrak{F}_{\rho_{\rm ex}}(\rho_L).
\end{split}
\end{equation*}
We thus obtain that  Lemma~\ref{l:futaki-computed} is true for ${\bf g}\geq 1$ too.
\begin{prop}\label{l:futaki-computed-1} Let $M=\PP(E)\to C$ with $E=\bigoplus_{k=0}^{\ell-1}U_k$ be a projectivisation of vector bundle over a curve $C$ of genus ${\bf g}\ge 1$ and $L\subset  U_0$ a sub bundle of one of the indecomposable components of $E$. The relative Donaldson--Futaki invariant $\mathfrak{F}_{\rho_{\rm ex}}(\rho_L)$ of the induced $\C^{\times}$ action $\rho_L$ on the central fibre $M_0=P(V)$ with respect to the test configuration of Lemma~\ref{l:test-configuration} and a polarization $\cL_{m}$  is positive multiple  of  \eqref{futaki-computed} where $\alpha_i, \alpha_{ij}$ and $\beta_i$ are given by Proposition~\ref{p:computation} with $c=m$,  and $A$ is the matrix  \eqref{matrixA}.
\end{prop}

\subsection{The case of an indecomposable bundle}\label{s:ell=1} This is the case $\ell=1$ in the setting  of the previous sections, i.e. $M=P(E)$ with  $E$ an indecomposable vector bundle over $C$, $L \subset E$ is a sub-bundle, and the central fibre of the test-configuration given by Lemma~\ref{l:test-configuration} is $M_0=\PP(F\oplus L)$ with $F=E/L$. In this case, the reduced automorphisms group ${\rm Aut}^{red}(M)$ then has rank $0$ and, therefore,   the relative  Donaldson--Futaki invariant reduces to the usual Futaki invariant $\mathfrak{F}(\rho_K)$ on the central fibre $M_0$. This is computed (algebraically) in  \cite[Theorem 5.13]{RT}  (see also Theorem~\ref{thm:csc} in the introduction) and it is shown that it is given by a positive multiple of $(\mu(E) - \mu(L))$. For the sake of
completeness, and to make a better contact between \cite{RT} and  the setting of this paper, we compute
below \eqref{futaki-computed}.

As the Donaldson--Futaki invariant \eqref{relative-futaki-3} reduces to the usual Futaki invariant (i.e. $f^M_{\rm ex}=0$ in this case), \eqref{futaki-computed} becomes
$$\mathring{\mathfrak{F}}_{\rho_{\rm ex}}(\rho_{L}) = (\alpha_0\beta_1-\alpha_1\beta_0),$$
so that, by Lemma~\ref{betakalpha0}, we obtain
\begin{equation}
\begin{split}
\mathring{\mathfrak{F}}_{\rho_{\rm ex}}(\rho_{L}) &= \frac{\pi_R^2}{r_V!(r_V+1)!} \left(2r_V(r_Ld_V-r_Vd_L)c +2 ({\bf g}-1-d_V)(r_Ld_V-r_Vd_L)\right)\\
                                                                                        &=\frac{2 \pi_R^2 r_V r_L}{r_V!(r_V+1)!} \big(\mu(E)- \mu(L)\big)\left(r_V c -d_V+ ({\bf g}-1)\right),
                                                                                        \end{split}
\end{equation}
where,  we recall, $\pi_R= (r_L-1)!(r_V-r_L -1)!$, $c$ is the constant determined by the polarization   $\cL_m$ on $M$ with $m=c$, and the expression $\left(r_V c -d_V+ ({\bf g}-1)\right)$  is manifestly positive by  \eqref{c-constraint}.

\begin{rem} In the case $\ell= 1$, the Donaldson--Futaki invariant on $M_0$ is computed
(by essentially the same construction)   in  \cite[Prop. 6]{ACGT-0}): up to a positive constant it is given by  the expression $-(\underline{\mathbf{c}}(x)/x)$,  
where $\underline{\mathbf{c}}(x)$ is related to the strictly negative function $c(s,x)$ appearing on page 575 of \cite{ACGT-0} by 
$\underline{\mathbf{c}}(x) := c\Big(\frac{2(1-{\bf g})}{(\mu(F)-\mu(L))}, x\Big),$ 
and $x= \frac{(\mu(F)-\mu(L)}{c}$. A straightforward computation shows that the expressions agree (up to multiplication of a positive constant).
\end{rem}

\subsection{The case $\ell=2$}\label{sum2} We now consider the case when $E=U_0 \oplus U_1$ is the direct sum of two \textit{indecomposable} bundles $U_0,U_1$. This is also equivalent to the assumption that the rank of the reduced group of automorphisms ${\rm Aut}^{red}(M)$ equals $1$.

In this case, the matrix $A$ induced by \eqref{matrixA} is a scalar,  $A=\Big(\frac{\alpha_0}{\alpha_{22}-\alpha_2^2}\Big)$ and   $\alpha_0\alpha_{22} -\alpha_2^2>0$ by Cauchy--Schwarz (this is also a positive multiple of the  $L^2$ square norm of the function $f_2$, see Sect.~\ref{s:calabi-anzats}). Consequently, we can restrict our attention on 
$$(\alpha_0\alpha_{22} -\alpha_2^2)\mathring{\mathfrak{F}}_{\rho_{\rm ex}}(\rho_L)=(\alpha_0\alpha_{22} -\alpha_2^2)(\alpha_0\beta_1 -\alpha_1\beta_0) - (\alpha_0\alpha_{12} - \alpha_1\alpha_2)(\beta_2\alpha_0 -\alpha_2\beta_0).$$

\begin{prop}\label{case2line}
 For any admissible K\"ahler class, the relative Donaldson--Futaki invariant $\mathfrak{F}_{\rho_{\rm ex}}(\rho_L)$ has the sign of $\mu_0-\mu_1$.
\end{prop}
\begin{proof}
 This is a direct computation of the quantities $\alpha_0,\alpha_1,\alpha_2,\beta_0,\beta_1,\beta_2,\alpha_{12},\alpha_{22}$ using Lemma \ref{techn}. It is obtained that
 \begin{align}\label{resleq2}(\alpha_0\alpha_{22} -\alpha_2^2)\mathring{\mathfrak{F}}_{\rho_{\rm ex}}(\rho_L)=\Gamma_0 (\mu_0-\mu_1),\end{align}
where
\begin{align*}
 \Delta_c&=r_Vc-d_V,\\
 \Gamma_0&=\Delta_c (\Delta_c + {\bf g}-1)(\Delta_c+ 2(c-\mu_2))\Gamma_1, \\
 \Gamma_1&=\frac{2\pi_R^4r_0r_1r_2(r_2-1)!^4(r_1-1)!^4(r_0-1)!^4}{(r_V+2)!(r_V+1)!r_V!^2}>0.
\end{align*}
Note that with \eqref{c-constraint}, one has $\Delta_c>0$. This finishes the proof. 
Of course the full expression of $\mathfrak{F}_{\rho_{\rm ex}}$ can be provided but it is particularly lengthy even in this case. \end{proof}

\begin{rem}
Let us mention at that stage that the classical Donaldson--Futaki invariant $\mathfrak{F}(\rho_L)$ is positive proportional to $\alpha_0\beta_1 -\alpha_1\beta_0$ which is
$$\alpha_0\beta_1 -\alpha_1\beta_0=\Gamma'_1(\Delta_c + {\bf g}-1)(\mu_2-\mu_1).$$
with $\Gamma_1'=\frac{2\pi_R^4r_1r_2(r_1-1)!^2(r_2-1)!^2(r_0-1)!^2}{(r_V+1)!r_V!}$. This points out that the computation of the classical Donaldson--Futaki invariant does not bring any information on the stability of $U_0$.
\end{rem}

\subsection{The general case and the proof of Theorem~\ref{thm:main}}\label{generalcase} 

With the notation of Section~\ref{secRFI}, we aim to compute the sign of following quantity
$$2\mathring{\mathfrak{F}}_{\rho_{\rm ex}}(\rho_{L}) = 2(\alpha_0\beta_1-\alpha_1\beta_0)- \sum_{j=2}^{\ell} a_j (\alpha_0\alpha_{j1}-\alpha_1\alpha_j).$$ 
Both the differential geometric and algebraic approaches  lead to the same difficulty of controlling the terms $a_j$. In order to do so, we are going to expand the unknowns $a_j$, solutions of \eqref{syst}, in Taylor series with respect to the variable $c$ (recall that $c=m$).
Our method consists in evaluating the quantities 
$$\Sigma_1=\sum_{j=2}^{\ell} a_jr_j, \hspace{1cm}\Sigma_2=\sum_{j=2}^{\ell} a_jd_j.$$
We write the Taylor expansions
\begin{align*}
       \Sigma_1&=\sum_{i=1}^\infty u_i c^{-i}, \\
       \Sigma_2&=\sum_{i=1}^\infty v_i c^{-i}. \\
\end{align*}
From the expression of the matrix $A$ in \eqref{matrixA} and the asymptotics above, it is clear that 
$a_j=\sum_{r=2}^\ell (A^{-1})_{rj}\left(\beta_r-\frac{\alpha_r}{\alpha_0}\beta_0\right)$ is at most $O(1/c)$. Consequently, $\Sigma_1$ and $\Sigma_2$ are at most 
$O(1/c)$ for $c\to +\infty$.
In order to ease the computations, we will assume $$d_V=0.$$ This is not a restrictive assumption. Actually, 
we can tensorize $V$  with the rational  line bundle $\cO(-\mu(V))_C$  and notice that this does not change the underlying variety $M$.  It only introduces a translation with  $-\mu(V)$ of the parameter $c=m$ of the polarizations on $M$.

\begin{notation}
 We denote $\mu_{01}=\frac{d_0+d_1}{r_0+r_1}$.
\end{notation}

\begin{lemma} \label{S}
 With notations as before, the terms $(u_i),(v_i)$ satisfy the system $(S)$
 {
\begin{equation*}
 (S)\begin{cases}
  u_1= 4{r_V^2}\mu_{01},\\
  (r_V+2)u_2 -2\mu_{01} u_1 
  -2 v_1=4 (r_V+2)({\bf g}-1)r_V \mu_{01},\\
   \text{and for }i\geq 1,\\
   (r_V+2)u_{i+2} -2\mu_{01} u_{i+1} -2v_{i+1}+ \frac{(r_V+2)}{(r_V+1)}\mu_{01} v_{i}=0\\
 \end{cases}
\end{equation*}
}
In particular, this provides the first term $u_1$ of the expansion of $\Sigma_1$.
\end{lemma}
\begin{proof}
Actually, the system \eqref{syst} implies that for $k\neq 0$, $k\neq 1$,
\begin{equation}
 \label{newsyst} \sum_{j=2}^{\ell} a_j(\alpha_{jk}\alpha_0- \alpha_j\alpha_k)=2(\beta_k\alpha_0-\beta_0\alpha_k).
\end{equation}
We expand each equation using the expressions of $(\alpha_{jk}\alpha_0- \alpha_j\alpha_k)$. Then we sum the equations from $k=2$ to $\ell$. This way, we obtain 
{
\begin{align}\label{linSigmas1}
 &[(r_V+2)c^2-2\mu_{01} c]\Sigma_1\\ \nonumber
 &+[-2c+\frac{(r_V+2)}{(r_V+1)}\mu_{01}]\Sigma_2 =2 \frac{r_V!(r_V+1)!(r_V+2)}{\pi_R^2 (r_0+r_1)} \sum_{k=2}^{\ell}(\beta_k\alpha_0-\beta_0\alpha_k).
\end{align}
}We use Lemma \ref{betakalpha0}. Eventually, we obtain the system by using the Taylor expansions of $\Sigma_1,\Sigma_2$.
\end{proof}

We explain briefly how the last result allows us to compute the expansions of $a_j$. From \eqref{newsyst}, at a fixed $k$, we have at first order in $c$,
\begin{align*}
 \gamma'_k a_k =& 2(\beta_k\alpha_0-\beta_0\alpha_k)- \sum_{j=2}^\ell a_j\gamma_{jk}\\
               =& -4\frac{\pi_R^2}{(r_V-1)!(r_V+1)!}r_Vd_kc\\
               &+ \frac{\pi^2_R}{(r_V+1)!^2(r_V+2)}(r_k(r_V+2)(r_V+1)c^2)\frac{u_1}{c}+O(1),
\end{align*}
where we have used the fact that $\sum_{j=2}^\ell a_j\gamma_{jk}$ can be written in terms of $\Sigma_1$ and $\Sigma_2$. This gives from Lemma \ref{gamma},
$$a_k={4r_V}(\mu_{01}-\mu_k)\frac{1}{c}+O(1/c^2).$$
From this expression, one can derive the first term of $\Sigma_2$, by summing, as
\begin{align}\label{v1}
{v_1} &=-{4r_V}\left( \frac{(d_0+d_1)^2}{(r_0+r_1)} + \sum_{j=2}^{\ell}\frac{d_j^2}{r_j}\right).
\end{align}
Back to $(S)$, we can deduce from $(u_1,v_1)$ the value of $u_2$ and apply the same trick recursively to deduce all the values of $(u_i),(v_i)$. 
This way we get
{
\begin{align}\label{u2}
 u_2=&-\frac{8r_V}{(r_V+2)}\left(\frac{(d_0+d_1)^2}{r_0+r_1}+\sum_{j=2}^{\ell}\frac{d_j^2}{r_j}\right) +{4r_V\mu_{01}}\left( {\bf g}-1+\frac{2r_V}{{r_V+2}}\mu_{01}\right)  
\end{align}
}
We are now ready to prove the main technical result of this section.
\begin{prop}\label{mainprop}
We have the following asymptotic expansion of the normalized relative Donaldson--Futaki invariant,
$$\mathring{\mathfrak{F}}_{\rho_{\rm ex}}(\rho_{L})=(\mu_0-\mu_1)(Fut_1c +Fut_2 + \frac{Fut_3}{c}+\frac{Fut_4}{c^2}+...)$$
where
$Fut_1>0$ and the $Fut_i$ are explicit for $i=1,2$.
\end{prop}

\begin{proof}
The computation of $\mathring{\mathfrak{F}}_{\rho_{\rm ex}}(\rho_{L})$ at first order in the $c$ variable depends only on the asymptotic expansion of $\Sigma_1$ at first order.
Using the expressions of $\alpha_0\beta_1-\alpha_1\beta_0$ and $\sum_{j=2}^{\ell}a_j\gamma_{j1}$, we obtain from Lemma \ref{gamma},  Lemma \ref{betakalpha0} and Lemma \ref{S} that provides the value of $u_1$, that 
 \begin{align*}\mathring{\mathfrak{F}}_{\rho_{\rm ex}}(\rho_{L})=&(\alpha_0\beta_1-\alpha_1\beta_0)-\frac{1}{2} \sum_{j=2}^{\ell} a_j(\alpha_{j1}\alpha_0-\alpha_1\alpha_j),\\
 =&(\alpha_0\beta_1-\alpha_1\beta_0)-\frac{1}{2} \sum_{j=2}^{\ell} a_j\gamma_{j1},\\
=& -2c\frac{\pi_R^2}{r_V!(r_V+1)!}r_V^2d_1+2c  \frac{\pi_R^2r_1}{(r_V+1)!r_V!} r_V^2 \mu_{01}    +O(1),\\
 =& 2c\frac{\pi_R^2 r_Vr_0r_1}{(r_V-1)!(r_V+1)!(r_0+r_1)}\left( \mu_0-\mu_1 \right)+O(1).\\ 
\end{align*}
Thus $Fut_1=\frac{2\pi_R^2r_Vr_0r_1}{(r_V-1)!(r_V+1)!(r_0+r_1)}>0$.
Using \eqref{u2} and \eqref{v1}, we obtain by brute force
$$\mathring{\mathfrak{F}}_{\rho_{\rm ex}}(\rho_{L})=(\mu_0-\mu_1)Fut_1c + (\mu_0-\mu_1)Fut_2 +O(1/c)$$ with explicitly
\begin{align*}Fut_2=& \frac{2r_0r_1\pi_R^2}{(r_V-1)!(r_V+2)!(r_0+r_1)}  [(r_V+2)({\bf g}-1) +2r_V\mu_{01}].
\end{align*}
Eventually, we justify that all the other terms of the expansion of $\alpha_0\mathfrak{F}_{\rho_{\rm ex}}$ are multiples of $(\mu_0-\mu_1)$.
In order to do so, we notice from Lemma \ref{betakalpha0} that these terms are given by the expansion of $-\sum_{j=2}^{\ell}a_j\gamma_{j1}=\sum_{i\geq -1}\frac{\widehat{Fut}_i}{c^i}$. This is given, up to a multiplicative factor $\frac{\pi_R^2(r_V+1)}{(r_V+1)!^2(r_V+2)}$, by
\begin{align*}
 r_1(r_V+2)\sum_{i\geq 1} \frac{u_{i+2}}{c^i}
 -2d_1\sum_{i\geq 1} \frac{u_{i+1}}{c^i}
   -2r_1\sum_{i\geq 1}\frac{v_{i+1}}{c^i} +d_1\frac{(r_V+2)}{(r_V+1)}\sum_{i\geq 1}\frac{v_i}{c^i}.
\end{align*}
Hence, we are doomed to check that
\begin{align*}
[r_1(r_V+2)(r_V+1)u_{i+2} -2(r_V+1)d_1u_{i+1} -2(r_V+1)r_1v_{i+1}+d_1(r_V+2)v_i],
\end{align*} 
are multiples of $(\mu_0-\mu_1)$ for $i\geq 1$. We use the last relationship given by $(S)$ in Lemma \ref{S}. This provides by replacing $u_{i+2}$
that 
$${\widehat{Fut}_{i+2}}=\frac{\pi_R^2}{(r_V+1)!^2(r_V+2)(r_0+r_1)}(2(r_V+1)u_{i+1}-v_i(r_V+2))(r_1d_0-d_1r_0).$$
and the conclusion holds as expected with $$Fut_{i+2}=\frac{\pi_R^2 r_1r_0}{2(r_V+1)!^2(r_V+2)(r_0+r_1)}(2(r_V+1)u_{i+1}-v_i(r_V+2)),$$
for $i\geq 1$.
\end{proof}

We refine Proposition \ref{mainprop} and show that the following holds.

\begin{prop}\label{proppositivity}
For $c>\max(\mu(V_i))$.
$$(Fut_1c +Fut_2 + \frac{Fut_3}{c}+\frac{Fut_4}{c^2}+...)>0.$$
In particular if the relative Donaldson--Futaki $\mathfrak{F}_{\rho_{\rm ex}}(\rho_L)$ is positive then $\mu_0-\mu_1>0$.
\end{prop}
\begin{proof}
 First we remark that from the proof of Proposition \ref{mainprop}, 
 \begin{align*}
  \sum_{i\geq 1}\frac{Fut_{i+2}}{c^{i}}=& \frac{\pi_R^2{r_1r_0}}{2(r_V+1)!^2(r_V+2)(r_1+r_0)}[2(r_V+1)c\sum_{i\geq 1}\frac{ u_{i+1}}{c^{i+1}}-(r_V+2)\sum_{i\geq 1}\frac{ v_{i}}{c^{i}}],   \\
  =& \frac{\pi_R^2{r_1r_0}}{2(r_V+1)!^2(r_V+2)(r_1+r_0)}[2(r_V+1)c\Sigma_1-(r_V+2)\Sigma_2]\\
  &-u_1 \frac{\pi_R^2{r_1r_0}}{r_V!(r_V+1)!(r_V+2)(r_1+r_0)}.    
 \end{align*}
 So we deduce using the expressions of $Fut_1,Fut_2$ computed in the previous proposition
  \begin{align}
\sum_{i\geq -1}\frac{Fut_{i+2}}{c^{i}}=&\frac{r_1r_0\pi_R^2}{2(r_1+r_0)(r_V+1)!^2(r_V+2)} \nonumber \\
 &\pushright{ \times[4r_V(r_V+1)(r_V+2)(r_Vc+{\bf g} -1)+2(r_V+1)c\Sigma_1-(r_V+2)\Sigma_2]}. \label{factorFut}                                      
  \end{align}
By the assumption $d_V=0$,  we  have $c>\frac{d_V}{r_V}=0$,  so that in order to prove the proposition, we only need to show the positivity of $$[4r_V(r_V+1)(r_V+2)(r_Vc+{\bf g} -1)+2(r_V+1)c\Sigma_1-(r_V+2)\Sigma_2].$$
 We have a linear relationship between $\Sigma_1$ and $\Sigma_2$ thanks to \eqref{linSigmas1}. We seek for a second linear relationship. 
 Firstly, from \eqref{newsyst} and Lemmas \ref{betakalpha0} and \ref{gamma}, we get using $d_V=0$,
 \begin{align*}
 \gamma'_ka_k =&\frac{\pi_R^2}{(r_V+1)!^2(r_V+2)}\\
               &\times \left[-4(r_V+1)(r_V+2)r_Vd_k(r_Vc + {\bf g} -1)\right. \\
               &\hspace{1cm} -(-r_k(r_V+1)(r_V+2)c^2+2d_k(1+r_V)c)\Sigma_1 \\
               &\hspace{1cm}\left.-(2r_k(1+r_V)c-(r_V+2)d_k)\Sigma_2\right].
 \end{align*}
On another side, from Lemma \ref{gamma}, 
\begin{align*}
\gamma'_k  
=&\frac{\pi_R^2}{(r_V+1)! (r_V+2)!}(r_V+1)r_k[r_V(r_V+2)c^2-2r_V\mu_kc].
\end{align*}
From these two previous equations, we obtain 
  \begin{align*}
  &a_k=\frac{-4(r_V+2)d_k\left(r_Vc +{\bf g} -1\right)}{c[r_k(r_V+2)c-2d_k] } +\frac{\Sigma_1}{r_V}- \frac{(2r_k(1+r_V)c-(r_V+2)d_k)\Sigma_2}{(r_V+1)r_Vc[r_k(r_V+2)c-2d_k]}.
 \end{align*}
We multiply this expression by $d_k$ and then sum. This provides
  \begin{align*}
   \Sigma_2=&-\sum_{k=2}^{\ell}\frac{4(r_V+2)d_k^2\left(r_Vc +{\bf g} -1\right)}{c[r_k(r_V+2)c-2d_k] }  +\frac{(d_V-d_0-d_1)}{r_V}\Sigma_1 \\
   &-\left( \sum_{k=2}^{\ell}\frac{d_k(2r_k(1+r_V)c-(r_V+2)d_k)}{(r_V+1)r_Vc[r_k(r_V+2)c-2d_k]}\right)\Sigma_2,
  \end{align*}
  i.e a second linear relationship between the unknowns $(\Sigma_1,\Sigma_2)$.
Hence, using \eqref{linSigmas1}, we can identify $\Sigma_2$ as
   \begin{align}\label{sigma2}
   &\Sigma_2=-\frac{\sum_{k=2}^{\ell}\frac{4(r_V+2)d_k^2\left(r_Vc +{\bf g} -1\right)}{c[r_k(r_V+2)c-2d_k] }+\frac{4(r_V+2)(d_0+d_1)\mu_{01}[cr_V+{\bf g} -1]}{ [(r_V+2)c^2-2\mu_{01}c]}}{\Delta_{\Sigma_2} },
\end{align}
where $\Delta_{\Sigma_2}$ is given by 
{
\begin{align*}
\Delta_{\Sigma_2}=& \sum_{k=2}^{\ell}\frac{d_k(2r_k(1+r_V)c-(r_V+2)d_k)}{(r_V+1)r_Vc[r_k(r_V+2)c-2d_k]} \\
&+\frac{r_V(r_V+2)(r_V+1)c^2-2(r_V+1)(r_V-r_0-r_1)\mu_{01}c-(r_V+2)(d_1+d_0)\mu_{01}}{r_V(r_V+1)c[(r_V+2)c-2\mu_{01}]}.
\end{align*}}
As we said previously, we are looking for the sign of 
\begin{align}
 [r_V(r_V+1)&(r_V+2)(r_Vc+2({\bf g}-1))+2(r_V+1)c\Sigma_1-(r_V+2)\Sigma_2] \nonumber\\
=&4r_V(r_V+1)(r_V+2)^2\frac{(cr_V+{\bf g} -1)c}{ [(r_V+2)c-2\mu_{01}]}\nonumber\\
&+(-\Sigma_2)\left(\frac{cr_V^2}{[(r_V+2)c-2\mu_{01}]}\right).  \label{form1}
\end{align}
We will show this is positive. Actually, the first term is positive because $c>0$ and $(r_0+r_1)c>(d_1+d_0)$. As ${\bf g}\geq 1$, and $r_kc>d_k$, the only difficulty is to show that $\Delta_{\Sigma_2}$ is positive, which implies easily $-\Sigma_2>0$. The proof of the proposition is complete with Lemmas \ref{lemma11} and \ref{lemma12} which exhaust all possible cases.
\end{proof}

\begin{lemma}\label{lemma11}
 Assume as above that $d_V=0$ and also that  $(d_1+d_0)\geq 0$. Then 
 $\Delta_{\Sigma_2}>0$.
\end{lemma}
\begin{proof}
We write
{
\begin{align*}
 \Delta_{\Sigma_2}  
 =&\sum_{k=2}^{\ell}\frac{(2r_k(1+r_V)c-(r_V+2)d_k)}{(r_V+1)r_Vc} B_{\Delta_{\Sigma_2}},
 \end{align*}
} where one has denoted
{
 \begin{align}
 B_{\Delta_{\Sigma_2}}=&\frac{d_k}{r_k(r_V+2)c-2d_k}\nonumber \\
 &+\frac{r_V(r_V+2)(r_V+1)c^2-2(r_V+1)(r_V-r_0-r_1)\mu_{01}c-(r_V+2)\mu_{01}(d_1+d_0)}{[2(r_V-r_1-r_0)(r_V+1)c+(r_V+2)(d_1+d_0)][(r_V+2)c-2\mu_{01}]}. \label{bracket}
 \end{align}

 }The factor term ${(2r_k(1+r_V)c-(r_V+2)d_k)}$ is positive. Next, we are interested in the term $B_{\Delta_{\Sigma_2}}$. Its denominator is positive as $d_1+d_0\geq 0$ and $r_kc>d_k$. 
Its numerator, after gathering the 2 terms, is given (up to a factor $(r_V+2)c$) by 
\begin{align}\label{numerBD2b}
 & 2(r_V+1)(r_Vr_k(c -\mu_{01}) +(r_kc-d_k)(r_0+r_1))c  +(c -\mu_{01})(d_0+d_1)(r_V+2)r_k  \\
& +(r_V+1)r_k[r_V^2- 2(r_0+r_1)]c^2 +  (d_1+d_0)[r_Vr_kc +(r_V+2)d_k].\nonumber
\end{align}
The first line is obviously positive. We only need to check that the last line is non negative. To do so, we write 
\begin{align*}
& (r_V+1)r_k[r_V^2- 2(r_0+r_1)]c^2 +  (d_1+d_0)[r_Vr_kc +(r_V+2)d_k]\\
& > (r_V+1)r_k[r_V^2- 2(r_0+r_1)]\mu_{01} c+  (d_1+d_0)[r_Vr_kc +(r_V+2)d_k],\\
& =(r_V+2)(d_1+d_0)\left( [(r_V+1)(r_V\frac{r_V}{r_1+r_0}-2) + {r_V}]\frac{r_kc}{r_V+2} +d_k\right).
\end{align*}
Now as $d_V=0$, $$d_k \geq -d_V^+, \,\text{ with }\,\, d_V^+= \sum_{j=0, d_j\geq 0}^\ell d_j\geq 0.$$ From the assumption on $c$, $c>\frac{d_V^+}{\sum_{j=0, d_j\geq 0} r_j}\geq \frac{d_V^+}{r_V-1}$ as there will be at least one subbundle $V_i \subset V$ of negative degree (otherwise $d_0=d_1 =\cdots =d_k=d_V^+=0$, the last line vanishes and we are done). We get since $r_k\geq 1$,
\begin{align*}
 &\left( [(r_V+1)(r_V\frac{r_V}{r_1+r_0}-2) + {r_V}]\frac{r_kc}{r_V+2} +d_k\right)\\
 &\hspace{2cm}\geq [(r_V+1)(r_V\frac{r_V}{r_V-1}-2) + r_V]\frac{d_V^+}{(r_V-1)(r_V+2)} -d_V^+,\\
 &\hspace{2cm}\geq\frac{2r_V}{(r_V+2)(r_V-1)^2}d_V^+\geq 0.\\
\end{align*}
\end{proof}

\begin{lemma}\label{lemma12}
 Assume as above that $d_V=0$ and also that $(d_1+d_0)\leq 0$. Then 
 $\Delta_{\Sigma_2}>0$.
\end{lemma}
\begin{proof}
 This is similar to the previous lemma. First, as $d_1+d_0\leq 0$, $c>\frac{d_V^+}{r_V-r_1-r_0}$ so,
\begin{align}\label{ineq01}
2(r_V-r_1-r_0)(r_V+1)c+(r_V+2)(d_1+d_0) > 2(r_V+1)d_V^+ -(r_V+2)d_V^+>0.
\end{align}
We consider the numerator of the $B_{\Delta_{\Sigma_2}}$ term given by \eqref{bracket}. Then \eqref{numerBD2b} can be rewritten 
\begin{align*}
 &r_k(r_V+1)[2(r_V-r_0-r_1)]c^2+2c(r_V+1)(r_1+r_0)(r_kc -d_k ) \\
 &-2\mu_{01}(r_V+1)r_k (  r_V -(r_0+r_1))c- r_k(d_0+d_1)(r_V+2)\mu_{01} \\ 
 &+(r_V+2)(d_1+d_0)d_k +r_k(r_V+1)r_V^2c^2.
\end{align*}
The first  terms is positive. The 2nd term is
\begin{align*}
-\mu_{01}r_k[2(r_V+1)(r_V-r_0-r_1)c +(d_1+d_0)(r_V+2)],
\end{align*}
which is positive from \eqref{ineq01}. The 3rd term is positive if $d_k\leq 0$. Let us assume $d_k>0$. Then, using the properties of $c$,
\begin{align*}
 (r_V+2)(d_1+d_0)d_k +(r_V+1)r_V(r_Vc)(r_k c) &\geq -(r_V+2)d_kd_V^+  + r_V(r_V+1)d_k d_V^+,\\
 &\geq 0.
\end{align*}
This concludes the proof.
 
\end{proof}

We obtain now the proof of Theorem \ref{thm:main}.

\begin{proof}[Proof of Theorem \ref{thm:main}]\label{proofmain} As we explained at beginning of Section~\ref{mainsect}, we only need to show that
the relative K-polystability of a K\"ahler class (corresponding to some value of the constant $c$) implies that each indecomposable factor $U_i$ of $E$  is a stable bundle. The test configuration $\mathcal M$ we defined is normal and not a product test configuration, so we get from Proposition~\ref{proppositivity}
$$\mu_0- \mu_1>0.$$ 
This means that  $V_1 = L$ does not destabilize $U_0$, i.e. $U_0$ is stable. The same reasoning by permuting  of the  $U_k$  concludes the proof. \end{proof}

\subsection{Proof of Corollary \ref{cor1}}\label{sectcor}

\begin{proof} As we have already mentioned in the introduction,  one direction of the relative Yau--Tian--Donaldson Conjecture (see Conjecture~\ref{c:YTD}),  namely that the existence of an extremal K\"ahler metric in $2\pi c_1(\cL)$ implies the relative K-polystability of $(M, \cL)$  is established (for any polarized variety) in \cite{gabor-stoppa}. We shall thus discuss  bellow the other direction of the conjecture in each of the cases (1)--(4) listed in the Corollary~\ref{cor1}.

\smallskip
(1)  Suppose $s=\ell =1$. Then  the  automorphisms group of $(M=\PP(E), \cL)$ has rank $0$ (see the beginning of Section~\ref{mainsect}) unless $E$ has rank $1$ and $M=\C\PP^1$. Thus, the relative Donaldson--Futaki invariant of $(M, \cL)$ coincides with the usual Donaldson--Futaki invariant and it follows from   \cite[Theorem~5.13]{RT}  (see also Section~\ref{s:ell=1}) that $E$ must be stable. By  the Narasimhan--Seshadri Theorem~\cite{NS},  $M=\PP(E)$ admits a CSC K\"ahler metric in each K\"ahler class, in particular in $c_1(\cL)$. 

Suppose $s=\ell=2$, i.e. $E=U_1\oplus U_2$. If $(M=\PP(E), \cL)$ is relative K-polystable then,  by Theorem~\ref{thm:main} ,  $U_1$ and $U_2$ must be stable.   In this case,  \cite[Theorem 1]{ACGT-0} shows that the existence of an extremal metric  in $\Omega= 2\pi c_1(\cL)$ is equivalent to the positivity of the extremal polynomial $F_{\Omega}(z)$ (see \cite[Definition 1]{ACGT-0}) over the interval $(-1,1)$. Furthermore, by \cite[Theorem 2]{ACGT-0}, the latter is satisfied  provided that the {\it K\"ahler} relative K-polystability of $(M, \cL)$ holds (see \cite{dervan-2} for a precise definition) whereas the relative K-polystability of $(M, \cL)$ insures only that $F_{\Omega}(z)>0$ on $(-1,1) \cap \Q.$ In the case when the base of $\PP(E)$ is a curve,  the explicit form  of the extremal polynomial $F_{\Omega}(z)$ is given at the beginning of Section 3.2 of \cite{ACGT-0}: it follows that for the rational class $\Omega=2\pi c_1(\cL)$ (i.e.  an admissible K\"ahler class corresponding to a rational parameter $x$) the constant  $c$ is also rational and one obtains that $F_{\Omega}(z)>0$ on $(-1,1)$ if and only if $F_{\Omega}(z)>0$ on $(-1,1)\cap \mathbb{Q}$. Consequently, one can improve slightly \cite[Theorem 2]{ACGT-0} if the base is a complex curve: the relative K-polystability of $(M, \cL)$ implies the positivity of the extremal polynomial $P_{\Omega}(z)$ and thus the existence of an extremal metric in $\Omega= 2\pi c_1(\cL)$.

\smallskip
(2) By the Narasimhan--Seshadri Theorem~\cite{NS},  $M=\PP(E)$ admits a CSC K\"ahler metric in each K\"ahler class, in particular in $c_1(\cL)$.

\smallskip
(3) By Theorem~\ref{thm:main},  in this case  $(M=\PP(E), \cL)$ cannot be relative K-polystable.

\smallskip
(4)  If $(M=\PP(E), \cL= \cL_{q,p})$ is relative K-polystable, by Theorem~\ref{thm:main},  $E=\bigoplus_{s=1}^{k} U_s$ with $U_i$ stable. In this case \cite[Theorem~3]{ACGT} or  the main result of \cite{Bro} implies  that there exists $c_0>0$ such that  {\it any}  K\"ahler class $2\pi c_1(\cL_{q,p})$ with $p/q>c_0$ is extremal (and  hence also relatively $K$-polystable). \end{proof}

 \section{Appendix}\label{appendix}
 
 \subsection{Integration over the simplex}
 
 \begin{lemma}\label{techn}
Let us fix the integers $m_k\geq 0$. We have
 \begin{equation*}
  \int_{\Delta} \prod_{k=0}^\ell (L_k(x))^{m_k}dv = \frac{m_0 ! m_1 ! ... m_\ell!}{(m_0+...+m_\ell+\ell)!}.
 \end{equation*}
\end{lemma}

\begin{proof}
 This is elementary but we include a proof as we could not find a reference in the literature. It is not difficult to see that 
 $$\frac{1}{T^{m+n+1}}\int_0^T y^m (T-y)^n dy = \int_0^1 x^m (1-x)^n dx=B(m+1,n+1),$$
 where $B(.,.)$ is the standard Beta function (also called the Euler integral of the first kind).
 Now, by integrating out one variable at each step, we obtain
 {\small
 \begin{align*}
  \int_{\Delta}& \prod_{k=0}^\ell (L_k(x))^{m_k}dv\\
  &=\int_0^1 \int_0^{1-x_1}...\int_0^{1-x_1-...-x_{\ell-1}}x_1^{m_1}...x_\ell^{m_\ell}(1-x_1...-x_\ell)^{m_0}dx_{\ell}...dx_2dx_1,\\
  &=\int_0^1 x_1^{m_1}\int_0^{1-x_1}...\int_0^{1-x_1-...-x_{\ell-1}}x_\ell^{m_\ell}(1-x_1...-x_\ell)^{m_0}dx_{\ell}...dx_2dx_1,\\
  &=\int_0^{1} x_1^{m_1}\int_0^{1-x_1} x_2^{m_2}...\int_0^{1-x_1-...-x_{\ell-2}}\hspace{-0.1cm} B(m_0+1,m_\ell+1)(x_{\ell-1}^{m_{\ell-1}}(1-x_1-...-x_{\ell-1})^{m_0+m_{\ell}+1}),\\
  &=B(m_0+1,m_\ell+1)B(m_{\ell-1}+1,m_0+m_\ell+2)\\
  &\hspace{1cm}\times \int_0^{1} x_1^{m_1} ... \int_0^{1-x_1-...-x_{\ell-3}} (x_{\ell-2}^{m_{\ell-2}}(1-x_1-...-x_{\ell-2})^{m_0+m_\ell+m_{\ell-1}+2}),
 \end{align*}
 }
and so on, till we get
 \begin{align*}
  \int_{\Delta}& \prod_{k=0}^\ell (L_k(x))^{m_k}dv\\
  &=B(m_0+1,m_\ell+1)B(m_0+m_\ell+2,m_{\ell-1}+1)...\\
  & \hspace{1cm}\times B(m_2+1,m_0+m_\ell+...+m_3+\ell-1)B(m_1+1,m_0+m_\ell+...+m_2+\ell),\\
  &=\frac{m_0!...m_\ell! }{(m_0+...+m_\ell+\ell)!},
 \end{align*}
 where for the last step we have used the classical relationship between the Beta and the Gamma function.
\end{proof}

We need the following lemma to treat the case of ranks equal to $1$.\begin{center}
 
\begin{lemma}\label{kappa}
The following relations hold:
\begin{align*}
\int_{\partial \Delta}p_c(x)d\sigma&=\frac{\pi_R}{(r_V-1)!}\left((r_V-1)\kappa c-\sum_{k=0}^\ell d_k\kappa_k\right),\\
\int_{\partial \Delta}x_j p_c(x)d\sigma&=\frac{\pi_R}{r_V!}\left(r_jr_V\kappa_j c- r_j\sum_{k=0}^\ell d_k\kappa_{k,j}-\kappa_jd_j\right).
\end{align*}
\end{lemma}
\begin{proof}
We start by computing $\int_{\partial \Delta}\prod_{k=0}^{\ell}L_k^{r_k-1}dv$. Denote $\Delta'_j$ the standard simplex obtained by freezing the coordinate $L_j(x)=1$. Then, applying Lemma \eqref{techn},
\begin{align*}
 \int_{\partial \Delta}\prod_{k=0}^{\ell}L_k^{r_k-1} d\sigma&= \sum_{j=0,r_j=1}^{\ell} \int_{\Delta'_j}\prod_{k=0,k\neq j}^{\ell}L_k^{r_k-1}dv,\\
 &= \sum_{j=0,r_j=1}^{\ell} \frac{\prod_{k=0,k\neq j}^{\ell}(r_k-1)!}{(\sum_{k\neq j}r_k -1)!},\\
 &= \sum_{j=0,r_j=1}^{\ell} \frac{\pi_R}{(\sum_{k=0}^\ell r_k -2)!},\\
 &= \frac{\pi_R}{(\sum_{k=0}^\ell r_k -2)!}\kappa.
\end{align*}
Now, we obtain
\begin{align*}
 \int_{\partial \Delta}p_c(x)d\sigma&=c\frac{\pi_R}{(r_V-2)!}\kappa -\sum_{k=0}^\ell  \mu_k \frac{\pi_R r_k}{(\sum_{k=0}^\ell r_k -1)!}\kappa_k,\\
 &=c\frac{\pi_R}{(r_V-2)!}\kappa -\frac{\pi_R}{(r_V-1)!}\sum_{k=0}^\ell  d_k\kappa_k,
\end{align*}
which leads to the first result. Now,
\begin{align*}
 \int_{\partial \Delta}x_jp_c(x)d\sigma&=c\frac{\pi_R}{(r_V-1)!}r_j\kappa_j -\sum_{k=0,k\neq j}^\ell  \mu_k r_kr_j\kappa_{k,j}\frac{\pi_R}{r_V!}-\mu_j\frac{\pi_R r_j(r_j+1)}{r_V!}\kappa_j,
\end{align*}
and this gives the second result as $\kappa_{j,j}=\kappa_j$.
\end{proof}    \end{center}

 \subsection{Chern characters of symmetric tensor powers of vector bundles}\label{chern}

In this section we gather some technical formulas.
 \begin{prop}\label{technique}
  Let $E$ a smooth vector bundle (or locally free sheaf) of rank $r_V$ over a smooth manifold. For $k\geq 1$, we denote the $S^k E$ the symmetric tensor power of order $k$ of $E$. Then,
  \begin{align*}
  rk(S^k E)=& \binom{r_V-1+k}{k},\\
  c_1(S^k E)=&\binom{r_V-1+k}{k-1}c_1(E),\\
  ch_2(S^k E)=&\binom{r_V+k}{k-1}ch_2(E) + \frac{1}{2}\binom{r_V-1+k}{k-2}c_1(E)^2,\\
  ch_3(S^k E)=&\frac{1}{6}\binom{r_V-1+k}{k-3} c_1(E)^3+\binom{r_V+k}{ k-2}c_1(E)ch_2(E)\\
  &+\left(\binom{r_V+1+k}{ k-1}+\binom{r_V+k}{k-2}\right)ch_3(E).\end{align*}
 \end{prop}

\begin{proof}
 This is done using splitting principle and symmetries. It can be checked easily that the formulas are correct for $E$ direct sum of 2 line bundles.
\end{proof}

\bigskip

{\small
\noindent {\bf Acknowledgments.} {\small \\ The first author was supported in part by an NSERC Discovery grant. He is grateful to the University  Aix-Marseille and to the Institute of Mathematics and Informatics of the Bulgarian Academy of Science for their hospitality and support during the preapration of this work. The second author is grateful to C. Tipler for useful conversations. His work has been carried out in the framework of the Labex Archim\`ede (ANR-11-LABX-0033) and of the A*MIDEX project (ANR-11-IDEX-0001-02), funded by the ``Investissements d'Avenir" French Government programme managed by the French National Research Agency (ANR). The second author was also partially supported by the ANR project EMARKS, decision No ANR-14-CE25-0010. 
}}

\bibliographystyle{plain}
\bibliography{biblio.bib}

\begin{thebibliography}{10}

\bibitem{ACGT-1}
V.~Apostolov, D.~M.~J. Calderbank, P.~Gauduchon, and C.~W.
  T{\o}nnesen-Friedman.
\newblock Hamiltonian 2-forms in {K}\"ahler geometry. {II}. {G}lobal
  classification.
\newblock {\em J. Differential Geom.}, 68(2):277--345, 2004.

\bibitem{ACGT-0}
V.~Apostolov, D.~M.~J. Calderbank, P.~Gauduchon, and C.~W.
  T{\o}nnesen-Friedman.
\newblock Hamiltonian 2-forms in {K}\"ahler geometry. {III}. {E}xtremal metrics
  and stability.
\newblock {\em Invent. Math.}, 173(3):547--601, 2008.

\bibitem{ACGT}
V.~Apostolov, D.~M.~J. Calderbank, P.~Gauduchon, and C.~W.
  T{\o}nnesen-Friedman.
\newblock Extremal {K}\"ahler metrics on projective bundles over a curve.
\newblock {\em Adv. Math.}, 227(6):2385--2424, 2011.

\bibitem{AB}
M.~F. Atiyah and R.~Bott.
\newblock The {Y}ang-{M}ills equations over {R}iemann surfaces.
\newblock {\em Philos. Trans. Roy. Soc. London Ser. A}, 308(1505):523--615,
  1983.

\bibitem{BV}
N.~Berline and M.~Vergne.
\newblock The equivariant index and {K}irillov's character formula.
\newblock {\em Amer. J. Math.}, 107(5):1159--1190, 1985.

\bibitem{BHJ}
S.~{Boucksom}, T.~{Hisamoto}, and M.~{Jonsson}.
\newblock {Uniform {K}-stability, {D}uistermaat-{H}eckman measures and
  singularities of pairs}.
\newblock {\em ArXiv}, 1504.06568, 2015.

\bibitem{Bro}
T~Br\"onnle.
\newblock Extremal {K}\"ahler metrics on projectivized vector bundles.
\newblock {\em Duke Math. J.}, 164(2):195--233, 2015.

\bibitem{cal-one}
E.~Calabi.
\newblock Extremal {K}\"ahler metrics.
\newblock In {\em Seminar on {D}ifferential {G}eometry}, volume 102 of {\em
  Ann. of Math. Stud.}, pages 259--290. Princeton Univ. Press, Princeton, N.J.,
  1982.

\bibitem{DV1}
Alberto Della~Vedova and Fabio Zuddas.
\newblock Scalar curvature and asymptotic {C}how stability of projective
  bundles and blowups.
\newblock {\em Trans. Amer. Math. Soc.}, 364(12):6495--6511, 2012.

\bibitem{dervan-2}
R.~Dervan.
\newblock Relative {K}-stability for {K}\"ahler manifolds.
\newblock {\em ArXiv}, 1611.00569, 2016.

\bibitem{dervan-1}
R.~Dervan.
\newblock Uniform stability of twisted constant scalar curvature {K}\" ahler
  metrics.
\newblock {\em Internat. Math. Res. Notices}, (15):4728--4783, 2016.

\bibitem{DR2016}
R.~Dervan and J.~Ross.
\newblock {K}-stability for {K}\"ahler manifolds.
\newblock {\em ArXiv}, 1602.08983, 2016.

\bibitem{Do2}
S.~K. Donaldson.
\newblock Scalar curvature and stability of toric varieties.
\newblock {\em J. Differential Geom.}, 62(2):289--349, 2002.

\bibitem{SD2016}
Z.~Sj\"ostr\"om Dyrefelt.
\newblock {K}-semistability of {CSCK} manifolds with transcendental cohomology
  class.
\newblock {\em ArXiv}, 1601.07659, 2016.

\bibitem{fujiki}
A.~Fujiki.
\newblock Remarks on extremal {K}\"ahler metrics on ruled manifolds.
\newblock {\em Nagoya Math. J.}, 126:89--101, 1992.

\bibitem{Fuj-Sch}
A.~Fujiki and G.~Schumacher.
\newblock The moduli space of extremal compact {K}\"ahler manifolds and
  generalized {W}eil-{P}etersson metrics.
\newblock {\em Publ. Res. Inst. Math. Sci.}, 26(1):101--183, 1990.

\bibitem{futaki}
A.~Futaki.
\newblock {\em K\"ahler-{E}instein metrics and integral invariants}, volume
  1314 of {\em Lecture Notes in Mathematics}.
\newblock Springer-Verlag, Berlin, 1988.

\bibitem{Fut}
A.~Futaki.
\newblock Asymptotic {C}how polystability in {K}\"ahler geometry.
\newblock In {\em Fifth {I}nternational {C}ongress of {C}hinese
  {M}athematicians. {P}art 1, 2}, volume~2 of {\em AMS/IP Stud. Adv. Math., 51,
  pt. 1}, pages 139--153. Amer. Math. Soc., Providence, RI, 2012.

\bibitem{gauduchon-book}
P.~Gauduchon.
\newblock Calabi's extremal metrics: {A}n elementary introduction.
\newblock {\em Book in preparation}.

\bibitem{hong}
Y-J. Hong.
\newblock Constant {H}ermitian scalar curvature equations on ruled manifolds.
\newblock {\em J. Differential Geom.}, 53(3):465--516, 1999.

\bibitem{KR1}
J.~Keller and J.~Ross.
\newblock A note on {C}how stability of the {P}rojectivisation of {G}ieseker
  stable bundles.
\newblock {\em Journal of Geom. Analysis}, pages 1--21, 2012.

\bibitem{Le-Sim}
C.~LeBrun and S.~R. Simanca.
\newblock Extremal {K}\"ahler metrics and complex deformation theory.
\newblock {\em Geom. Funct. Anal.}, 4(3):298--336, 1994.

\bibitem{lejmi}
M.~Lejmi.
\newblock Extremal almost-{K}\"ahler metrics.
\newblock {\em Internat. J. Math.}, 21(12):1639--1662, 2010.

\bibitem{Li-Xu}
C.~Li and C.~Xu.
\newblock Special test configuration and {K}-stability of {F}ano varieties.
\newblock {\em Ann. of Math. (2)}, 180(1):197--232, 2014.

\bibitem{LSeyy}
Z.~Lu and R.~Seyyedali.
\newblock Extremal metrics on ruled manifolds.
\newblock {\em Adv. Math.}, 258:127--153, 2014.

\bibitem{Miyaoka}
Y.~Miyaoka.
\newblock The {C}hern classes and {K}odaira dimension of a minimal variety.
\newblock In {\em Algebraic geometry, {S}endai, 1985}, volume~10 of {\em Adv.
  Stud. Pure Math.}, pages 449--476. North-Holland, Amsterdam, 1987.

\bibitem{NR}
M.~S. Narasimhan and S.~Ramanan.
\newblock Deformations of the moduli space of vector bundles over an algebraic
  curve.
\newblock {\em Ann. Math. (2)}, 101:391--417, 1975.

\bibitem{NS}
M.~S. Narasimhan and C.~S. Seshadri.
\newblock Stable and unitary vector bundles on a compact {R}iemann surface.
\newblock {\em Ann. of Math. (2)}, 82:540--567, 1965.

\bibitem{Odaka2}
Y.~Odaka.
\newblock A generalization of the {R}oss-{T}homas slope theory.
\newblock {\em Osaka J. Math.}, 50(1):171--185, 2013.

\bibitem{Odaka1}
Y.~Odaka.
\newblock The {GIT} stability of polarized varieties via discrepancy.
\newblock {\em Ann. of Math. (2)}, 177(2):645--661, 2013.

\bibitem{rollin}
Y.~Rollin.
\newblock K-stability and parabolic stability.
\newblock {\em Adv. Math.}, 285:1741--1766, 2015.

\bibitem{RT}
J.~Ross and R.~Thomas.
\newblock An obstruction to the existence of constant scalar curvature
  {K}\"ahler metrics.
\newblock {\em J. Differential Geom.}, 72(3):429--466, 2006.

\bibitem{stoppa}
J.~Stoppa.
\newblock K-stability of constant scalar curvature {K}\"ahler manifolds.
\newblock {\em Adv. Math.}, 221(4):1397--1408, 2009.

\bibitem{stoppa-corrigendum}
J.~Stoppa.
\newblock A note on the definition of {K}-stability.
\newblock {\em ArXiv}, 1111.5826, 2011.

\bibitem{gabor-stoppa}
J.~Stoppa and G.~Sz{\'e}kelyhidi.
\newblock Relative {K}-stability of extremal metrics.
\newblock {\em J. Eur. Math. Soc. (JEMS)}, 13(4):899--909, 2011.

\bibitem{gabor}
G.~Sz{\'e}kelyhidi.
\newblock Extremal metrics and {$K$}-stability.
\newblock {\em Bull. Lond. Math. Soc.}, 39(1):76--84, 2007.

\bibitem{Sz1}
G.~Sz{\'e}kelyhidi.
\newblock The {C}alabi functional on a ruled surface.
\newblock {\em Ann. Sci. \'Ec. Norm. Sup\'er. (4)}, 42(5):837--856, 2009.

\bibitem{gabor-book}
G.~Sz\'ekelyhidi.
\newblock {\em An introduction to extremal {K}\"ahler metrics}, volume 152 of
  {\em Graduate Studies in Mathematics}.
\newblock American Mathematical Society, Providence, RI, 2014.

\bibitem{Sz2015}
G.~Sz{\'e}kelyhidi.
\newblock Filtrations and test-configurations.
\newblock {\em Math. Ann.}, 362(1-2):451--484, 2015.
\newblock With an appendix by Sebastien Boucksom.

\bibitem{Tian2}
G.~Tian.
\newblock K\"ahler-{E}instein metrics with positive scalar curvature.
\newblock {\em Invent. Math.}, 130(1):1--37, 1997.

\bibitem{Tian}
G.~Tian.
\newblock Bott-{C}hern forms and geometric stability.
\newblock {\em Discrete Contin. Dynam. Systems}, 6(1):211--220, 2000.

\bibitem{Wang4}
X.~Wang.
\newblock Height and {GIT} weight.
\newblock {\em Math. Res. Lett.}, 19(4):909--926, 2012.

\end{thebibliography}

\end{document}